\documentclass[reqno]{amsart}
\usepackage{graphicx}
\usepackage{epsfig}
\usepackage{amssymb}
\usepackage{xcolor}
\usepackage[
backend=biber,
style=alphabetic,
sorting=nyt
]{biblatex}
\usepackage{amsfonts}
\usepackage{dsfont}
\usepackage{amsmath}
\usepackage{stmaryrd}
\usepackage{hyperref}
\usepackage{subfigure, verbatim}
\newtheorem{theorem}{Theorem}
\newtheorem*{theorem*}{Theorem}
\newtheorem{lemma}{Lemma}
\newtheorem{proposition}{Proposition}[section]
\theoremstyle{definition}
\newtheorem{definition}{Key Proposition}

\newtheorem{corollary}{Corollary}

\theoremstyle{remark}
  
\numberwithin{equation}{section}
\newcommand\numberthis{\addtocounter{equation}{1}\tag{\theequation}}
\theoremstyle{remark}

\newcommand{\Prob}{\mathbb{P}}
\newcommand{\E}{\mathbb{E}}

\newcommand{\beq}{\begin{equation}}
\newcommand{\eeq}{\end{equation}}
\newcommand{\beqn}{\begin{equation*}}
\newcommand{\eeqn}{\end{equation*}}
\newcommand{\bea}{\begin{eqnarray}}
\newcommand{\eea}{\end{eqnarray}}
\newcommand{\bean}{\begin{eqnarray*}}
\newcommand{\eean}{\end{eqnarray*}}

\newcommand{\be}{\begin{enumerate}}
\newcommand{\ee}{\end{enumerate}}
\newcommand{\bi}{\begin{itemize}}
\newcommand{\ei}{\end{itemize}}
\newcommand{\bd}{\begin{description}}
\newcommand{\ed}{\end{description}}

\begin{document}
\title[Bounds for weighted sums of Rademacher multiplicative functions]{Almost sure bounds for weighted sums of Rademacher multiplicative functions}
\author{Christopher Atherfold}
\begin{abstract}
    We prove that when $f$ is a Rademacher random multiplicative function for any $\epsilon>0$, then $\sum_{n \leqslant x}\frac{f(n)}{\sqrt{n}} \ll (\log\log(x))^{3/4+\epsilon}$ for almost all $f$. We also show that there exist arbitrarily large values of $x$ such that $\sum_{n \leqslant x}\frac{f(n)}{\sqrt{n}} \gg (\log\log(x))^{-1/2}$. This is different to what is found in the Steinhaus case, this time with the size of the Rademacher Euler product making the multiplicative chaos contribution the dominant one. We also find a sharper upper bound when we restrict to integers with a prime factor greater than $\sqrt{x}$, proving that $\sum_{\substack{n \leqslant x \\ P(n) > \sqrt{x}}}\frac{f(n)}{\sqrt{n}} \ll (\log\log(x))^{1/4+\epsilon}$. 
\end{abstract}
\maketitle
\section{Introduction}
One of the central problems studied in analytic number theory is the Riemann Hypothesis for the Riemann zeta function. One of the many equivalent conditions to the Riemann hypothesis is M\"{o}bius cancellation, which is stated as for $\mu(n)$ the M\"{o}bius function, and any $\epsilon>0$, we have that
\begin{equation*}
    \sum_{n \leqslant x}\mu(n) \leqslant x^{1/2+\epsilon}.
\end{equation*}
The structure of $\mu(n)$ is rather mysterious and it is very difficult to understand the internal cancellations within this sum. A possible model for the partial sums of the Mobius function are partial sums of Rademacher random multiplicative functions, introduced by Wintner. The Rademacher multiplicative function $f(n)$ is defined as independent Rademacher random variables on the primes $p$, and is extended to the squarefree integers multiplicatively (it takes values of $0$ on non-squarefree numbers). In his 1944 paper \cite{wintner1944}, Wintner proved that for any $\epsilon > 0$, one almost surely has
\begin{equation*}
    \sum_{n \leqslant x}f(n) = O(x^{1/2+\epsilon})
\end{equation*}
and by partial summation we deduce that
\begin{equation*}
    M_f(x) := \sum_{n \leqslant x}\frac{f(n)}{\sqrt{n}} = O(x^\epsilon)
\end{equation*}
almost surely. The unweighted version of this problem has attracted a lot of attention recently, with the best known upper and lower bounds achieved by Caich and Harper respectively \cite{caich23}, \cite{harper23}, where they have shown that for all $\epsilon > 0$, one has
\begin{equation*}
    \sum_{n \leqslant x}f(n) \ll \sqrt{x}(\log\log(x))^{3/4+\epsilon}
\end{equation*}
and that there exist arbitrarily large values of $x$ such that
\begin{equation*}
    \sum_{n \leqslant x}f(n) \geqslant \sqrt{x}(\log\log(x))^{1/4-\epsilon}.
\end{equation*}
Note that the above is a slight simplification of Harper's result (which is slightly sharper than the version we state here). In view of these results, it seems unlikely that this particular model reflects the truth of detecting large fluctuations of partial sums of the Mobius function. Much work has also been developed to establish sharp bounds on the moments of sums of these multiplicative functions, with the sharp cases being identified by Harper \cite{aharper201}, \cite{harper202}. Finally, we note that following the groundbreaking work Gorodetsky and Wong determining the distribution for sums of Steinhaus random multiplicative functions over integers \cite{gorodetsky2025limiting}, one could be able to refine this to give a quantitative central limit theorem, which could give a more standard procedure to determine the large fluctuations of such sums.\par
The case of studying weighted sums of random multiplicative functions has focused on the case when $f$ is a Steinhaus multiplicative functions thus far. These are defined by choosing $f(p)$ as random variables which are uniformly distributed on the complex unit circle and $f(n)$ is defined completely multiplicatively (we add at this point that sums of Steinhaus random multiplicative functions might be a good model for studying $\mu(n)n^{it}$ for example). Steinhaus multiplicative functions will be referred to as $f_{st}$ from now on. Moments of $M_{f_{st}}(x)$ have been studied as a possible model for moments of $\zeta(1/2+it)$, which has been covered first in the work of Conrey and Gamburd \cite{conrey2006pseudomoments}, then Bondarenko, Heap and Seip \cite{bondarenko2015inequality} and Gerspach \cite{gerspach20}. However, the study of large fluctuation bounds for these weighted sums was first appeared in a paper of Aymone, Heap and Zhao \cite{aymone2021partial}. In a more recent paper, Aymone \cite{aymone24} noted that applying partial summation to Caich's upper bound, one has that
\begin{equation*}
    M_f(x) \ll (\log(x))^2.
\end{equation*}
After this, Hardy refined some of the methods used in their paper to find almost surely sharp bounds for $M_{f_{st}}(x)$ \cite{shardy23}. Hardy found that for all $\epsilon > 0$, one almost surely has that
\begin{equation*}
    M_{f_{st}}(x) \ll \exp\left((1+\epsilon)\sqrt{\log_2(x)\log_4(x)}\right) 
\end{equation*}
and for any $\epsilon>0$, one has
\begin{equation*}
    \limsup_{x \to \infty}\frac{|M_{f_{st}}(x)|}{\exp\left((1-\epsilon)\sqrt{\log_2(x)\log_4(x)}\right)} = \infty,
\end{equation*}
where $\log_k(x)$ denotes the $k$-fold iterated logarithm. As noted in his paper, these bounds resemble the law of the iterated logarithm very strongly. In this paper, we will combine the methods of Caich and Hardy to prove an upper bound for $M_f(x)$.
\begin{theorem}
    For any $\epsilon>0$ we have almost surely
    \begin{equation}
        M_f(x) \ll (\log\log(x))^{3/4+\epsilon}.
    \end{equation}\label{ubtheorem}
\end{theorem}
When we restrict to integers with a single large prime factor (ie $P(n) > \sqrt{x}$ where $P(n)$ is the largest prime factor to divide $n$), we are able to get a conjecturally sharp upper bound (by adapting the work of Mastrostefano \cite{mastrostefano2022almost}).
\begin{theorem}
    For any $\epsilon > 0$ we have almost surely
    \begin{equation}
        \sum_{\substack{n \leqslant x \\ P(n) > \sqrt{x}}}\frac{f(n)}{\sqrt{n}} \ll (\log\log(x))^{1/4+\epsilon}.
    \end{equation} \label{sharp_restricted_ub}
\end{theorem}
We delay the proof of Theorem \ref{sharp_restricted_ub} until Section \ref{mastro_variant}. The first bound is not sharp, but the author conjectures the second bound is. The reason for the first upper bound not being sharp will be explained later in the paper. We will also prove an almost sure lower bound for $M_f(x)$ using an adapted version of Hardy's method.
\begin{theorem}
    There exist arbitrarily large $x$ such that
    \begin{equation}
        |M_f(x)| \gg \left(\log\log(x)\right)^{-1/2}. 
    \end{equation} \label{lbtheorem}
\end{theorem}
It seems likely that one can prove a lower bound of
$$\gg (\log\log(x))^{1/4-\epsilon}$$
when restricting to integers whose largest prime factor is greater than $\sqrt{x}$ by appealing to the work of Harper with some changes \cite{harper23}, which would obtain sharp bounds in this case. \par
In Hardy's work, he showed that a main contribution to this sum came from
\begin{equation*}
    F_y^{st}(s) := \prod_{p \leqslant y}\left(1-\frac{f_{st}(p)}{p^s}\right)^{-1} \approx \exp\left(\sum_{p \leqslant y}\frac{f_{st}(p)}{p^s} + \sum_{p \leqslant y}\frac{(f_{st}(p))^2}{2p^{2s}}\right)    
\end{equation*}
when $s=1/2$ (in other words, the large fluctuations come from the size of this random Euler product. The Rademacher Euler product is given by
\begin{equation*}
    F_y(s) := \prod_{p \leqslant y}\left(1+\frac{f(p)}{p^s}\right) \approx \exp\left(\sum_{p \leqslant y}\frac{f(p)}{p^s} - \sum_{p \leqslant y}\frac{1}{2p^{2s}}\right).    
\end{equation*}
Geis and Hiary showed that for $s = \sigma+it$
\begin{equation}
    F(\sigma) = \prod_p \left(1+\frac{f(p)}{p^\sigma}\right) \to 0 \label{unshifted_euler}
\end{equation}
as $\sigma \to 1/2^+$ \cite{geishiary23} (in this paper we inspect this convergence in a more quantitative manner). If the heuristic in the Steinhaus case persisted, then this would suggest large fluctuations of the size
\begin{equation}
M_f(x) \overset{?}{\approx} \frac{\exp(\sqrt{\log_2(x)\log_4(x)})}{\sqrt{\log(x)}}. \label{bad_predict}
\end{equation}
Our upper bound for large fluctuations is significantly larger than this. One way to explain this is to consider $s$ away from the real line (introducing the imaginary part of $s$) then the contribution from this term decays quite rapidly (it acts a lot like the covariance of a log-correlated field). Such a prediction in equation (\ref{bad_predict}) would assume that the largest contribution is coming from the Rademacher Euler product when $t=0$, which we will find it does not, and it is the multiplicative chaos contribution which is dominant. For a more detailed discussion of this, see the very recent preprint of Hardy comparing the differences in distribution between the Rademacher and Steinhaus cases \cite{hardy2025distribution}, where this deterministic term considerably complicates the covariance computations. This is also the first result the author is aware of where the Rademacher model gives more cancellation than the Steinhaus model (for example the moments of partial sums of unweighted Rademacher random multiplicative functions are significantly larger than those in the Steinhaus case when considering moments of order greater than $2\phi$, where $\phi$ is the golden ratio).\par
The key objects of study are Euler products like equation (\ref{unshifted_euler}) and integrals of the form
\begin{equation}
    \int_{N-1/2}^{N+1/2}|F_y(1/2+\sigma+it)|^2dt \label{rep}
\end{equation}
for $\sigma > 0$ and $N \in \mathbb{Z}$. Consequently, we will draw upon some of the multiplicative chaos results derived by Harper (in particular Multiplicative Chaos Result 3 in \cite{harper23} and Proposition 6 in \cite{aharper201}). These manipulations explain why the upper bound in Theorem \ref{ubtheorem} are so much larger than those that would be predicted by the Law of Iterated Logarithm in this case, which is further remarked upon by Harper in his high moments paper for sums of random multiplicative functions \cite{harper202}. In particular, Harper comments on a transition in the behaviour of the Rademacher Euler product according to the value of $t$, noting that when considering $t=0$, then equation (\ref{rep}) resembles averaging over an orthogonal family of Dirichlet characters, whereas when $t \approx 1$, equation (\ref{rep}) is like averaging over a unitary family of Dirichlet characters.\par
The proof of Theorem \ref{ubtheorem} highlights the fact that the multiplicative chaos features of $M_f(x)$ are responsible for the size of its large fluctuations. Given how well the methods of Caich's proof in the unweighted case can be adapted to understanding $M_f(x)$, then the author believes that any improvements on this method could be extended to further work on $M_f(x)$ (as well as improvements to the work in Section \ref{ce} concerning the expectations of integrals of Rademacher Euler products). As for the lower bound, the method we use exploits the fact there is a deterministic contribution to $|M_f(x)|^2$. To improve this, it is likely we will need a new method to prove the lower bound. One reason is that it is very difficult to separate the smooth contribution in the weighted case, whereas it is discarded rather simply when we have no extra weight. We will encounter this problem even in this paper in the proof of Lemma \ref{PB_0k}. \par
Given the behaviour of the lower bound in the unweighted case and the upper bound we obtain in Theorem \ref{sharp_restricted_ub} (among other things), the author would conjecture that for any $\epsilon>0$, one has almost surely
$$M_f(x) \ll (\log\log(x))^{1/4+\epsilon}$$
and for any $V(x) \to \infty$, one can find arbitrarily large values of $x$ such that
$$|M_f(x)| \geqslant \frac{(\log\log(x))^{1/4}}{V(x)}.$$
However, this will not be an easy task, particularly in the lower bound since one will have to find a way to lower bound or discard the part of the sum where the $n$ do not have a large prime factor (which as previously mentioned is rather difficult given the weight factor we have). Proving a lower bound with the restriction $P(n) > \sqrt{x}$ should be doable by adapting the Harper's large fluctuations result \cite{harper23}. Regardless, it appears that $M_f(x)$ should diverge, with further evidence being provided by Aymone where he shows that $M_f(x)$ has infinitely many sign changes \cite{aymone24}. If one finds a divergent lower bound, then one could use it to prove a quantitative version of the result of Aymone. Another approach to this problem would be to implement the methods detailed in a recent paper of Klurman--Munsch--Lamzouri on sign changes in partial sums \cite{klurman2024sign}.\par
Another motivation to study $M_f(x)$ is its connection to forming large fluctuation bounds for unweighted sums of completely multiplicative Rademacher functions, which will be denoted by $f^*(n)$. These are defined by setting $f^*(p)$ as independent Rademacher random variables on primes $p$ and defining $f^*(n)$ completely multiplicatively. We then see that
\begin{equation}
    \sum_{n \leqslant x}f^*(n) = \sum_{ab^2 \leqslant x}f^*(ab^2) = \sum_{a \leqslant x}f(a) \left(\sum_{b^2 \leqslant x/a} 1\right) = \sqrt{x}\sum_{a \leqslant x}\frac{f(a)}{\sqrt{a}} - \sum_{a \leqslant x}\left\{\sqrt{\frac{x}{a}}\right\}f(a). \label{heur}
\end{equation}
and for a similar weighted version
\begin{equation}
    \sum_{n \leqslant x}\frac{f^*(n)}{\sqrt{n}} = \sum_{ab^2 \leqslant x}\frac{f^*(ab^2)}{b\sqrt{a}} = \sum_{b \leqslant \sqrt{x}}\frac{1}{b} \left(\sum_{a \leqslant x/b^2} \frac{f(a)}{\sqrt{a}}\right)\label{heur2}
\end{equation}
where $f(n) = f^*(n)$ when $n$ is squarefree and is $0$ otherwise. In particular, we can see that understanding the large fluctuations of $\sum_{n \leqslant x}f^*(n)$ is intertwined with understanding the large fluctuations of $M_f(x)$. The author does not know how one would attempt to handle the second sum because $\{\sqrt{x/n}\}$ is known to have particularly difficult behaviour to control, on top of the fact it would break the multiplicative structure which current proofs use to study $\sum_{n \leqslant x}f(n)$. While not as significant to understand, equation (\ref{heur2}) is enough to provide an almost sure upper bound when combined with Theorem \ref{ubtheorem}.
\begin{corollary}
    Let $\epsilon>0$. Then we have almost surely,
    $$\sum_{n \leqslant x}\frac{f^*(n)}{\sqrt{n}} \ll \log(x)(\log\log(x))^{3/4+\epsilon}.$$ \label{weighted_CR}
\end{corollary}
\begin{proof}
    In view of equation (\ref{heur2}), it suffices to simply input the bound we found in Theorem \ref{ubtheorem}. Then we have that by the triangle inequality
    $$\left|\sum_{n \leqslant x}\frac{f^*(n)}{\sqrt{n}}\right| \leqslant \sum_{b \leqslant \sqrt{x}}\frac{1}{b} \left|\sum_{a \leqslant x/b^2} \frac{f(a)}{\sqrt{a}}\right|.$$
    By Theorem \ref{ubtheorem}, the inner sum is almost surely $\ll (\log\log(x/b^2))^{3/4+\epsilon} \ll (\log\log(x))^{3/4+\epsilon}$. Summing over $b$ yields the corollary. 
\end{proof}
This bound certainly looks close to what one might expect for such a quantity since from the contribution of the squares less than $x$ alone is $\frac{1}{2}\log(x)$ (and it seems unlikely that the randomness would conspire to reduce this below $\log(x)$ order of magnitude). A lower bound of size $c(\log(x))^{1/2}$ for some $c>0$ can be proved as well using some tools from Gaussian processes as well (by following the proof of the lower bound in \cite{shardy23}). This being significantly larger than Theorem \ref{ubtheorem} is not surprising since $f^*(n)$ has deterministic contributions from the squares, as well as the random behaviour from the non-squares. Recently, Angelo proved that this sum sign changes infinitely often too, answering a question of Aymone \cite{angelo2024sign} (quantitative results of this kind could also be found using ideas from \cite{klurman2024sign}). One can apply partial summation to this, and find that for any $\epsilon>0$, then almost surely
$$\sum_{n \leqslant x}f^*(n) \ll \sqrt{x}(\log(x))^{1+\epsilon}.$$
In view of equation (\ref{heur}), hopefully one might be able to reduce the power of $\log(x)$. \par
In upcoming work of the author, he shows that for $q \in (0,1/2)$, then
$$\mathbb{E}\left|\sum_{2 \leq n \leq x}\frac{f(n)}{\sqrt{n}}\right|^{2q} \asymp (\log\log(x))^{-q/2},$$
which provides further evidence for the conjecture on the size of the large fluctuations of the weighted Rademacher sum. 
An application of all of these bounds would be for calculating the moments of sums of quadratic Dirichlet characters. Harper and Hussain successfully applied this principle by relating character sums to sums of Steinhaus random multiplicative functions in \cite{harper2023typicalsizecharacterzeta} and \cite{hussain22} (these were in rather different contexts). Later, Hussain and Lamzouri extended this principle to sums of Legendre symbols in \cite{hussain2023limiting}, where the limiting object featured the completely multiplicative Rademacher random multiplicative functions as well. These ideas also manifested in the work of Klurman--Lamzouri--Munsch when considering Fekete polynomials \cite{klurman2023l_q}.
\subsection{Ideas for Proofs}
The proof of Theorem \ref{ubtheorem} follows the work of Caich on the best known almost sure upper bound \cite{caich23}. The spirit of the proof can be traced back to Hal\'{a}sz \cite{halasz83}. Since we use the first Borel--Cantelli lemma to do this, it is natural to reduce this problem to proving our bounds on a suitable set of test points $(x_i)$ given that $|M_f(x_i)-M_f(x_{i-1})|$ does not grow too much. The later was proved by Lau--Tenenbaum--Wu in their paper \cite{WTL13}. Since it is now sufficient to understand only the size of $M_f(x_i)$, one can add some splitting according to the largest prime dividing $n$. For a suitable strictly increasing sequence $(y_j)_{0 \leqslant j \leqslant J}$, we have the splitting
\begin{equation}
    M_f(x_i) = \sum_{\substack{n \leqslant x_i\\P(n) \leqslant y_0}}\frac{f(n)}{\sqrt{n}} + \sum_{y_0 < p \leqslant y_J}\frac{f(p)}{\sqrt{p}}\sum_{\substack{n \leqslant x_i/p \\ P(n) < p}}\frac{f(n)}{\sqrt{n}} \label{prime_split}
\end{equation}
where $J$ is the smallest $j$ such that $y_j > x_i$). Caich then further splits the second term on the right hand side into various cases including splitting on the size of $j$ and over different prime ranges (this involves some of the deductions which appear in Harper's low moments paper \cite{aharper201}). The purpose of splitting the primes further than what is done in Lau--Tenenbaum--Wu is that unless the primes are suitably large in terms of $x_i$, then the variance of
\begin{equation}\sum_{y_{j-1} < p \leqslant y_j}\frac{f(p)}{\sqrt{p}}\sum_{\substack{n \leqslant x_i/p \\ P(n) < p}}\frac{f(n)}{\sqrt{n}} \label{prime_increment}
\end{equation}
is quite a bit smaller than what is contributing to the size of the large fluctuations (this idea is covered in a lot more detail in \cite{harper23}). Consequently, we see that it is only a small portion of the $j$ in equation (\ref{prime_split}) where the quantity in equation (\ref{prime_increment}) is large (in particular when $y_j$ is large enough in terms of $x_i$). The other important observation is that equation (\ref{prime_split}) is a sum of martingale differences which was first utilised by Mastrostefano \cite{mastrostefano2022almost}. Both proofs use the same sparse set of test points $(X_\ell)_{\ell \geqslant 1}$ where one takes the supremum of the $x_i$ over (which is why we need to the stronger martingale techniques used in this proof). This is significant since it allows Caich to use a very strong upper bound derived from the Azuma--Hoeffding inequality which combined with the better low moments estimates, a suitable set of sparser test points to implement this saving and other things allow him to save $(\log\log(x))^{5/4}$ compared to the previous work of Basquin and Lau--Tenenbaum--Wu \cite{B17}.\par
The application of Caich and Mastrostefano's ideas is particularly powerful in our case of $M_f(x)$ since only having to rely on the variance
$$\sum_{y_0 < p \leqslant y_0}\frac{1}{p}\E\left(\left|\sum_{\substack{n \leqslant x_i/p \\ P(n) < p}}\frac{f(n)}{\sqrt{n}}\right|^2\right)$$
means we do not have to rely on high moment bounds as much. \par
Hardy's work \cite{shardy23} is based off Lau--Tenenbaum--Wu as well as some other results involving strong bounds on the random Euler products of Steinhaus multiplicative functions and their integrals. It relies on manipulating the weighted sums into integrals of random Euler products, and splitting the range of integration to bound each contribution precisely. Hardy uses four sets of test points to average over in his paper to gain a very strong bound on the magnitude of the Steinhaus Euler product. In this paper, we implement a similar approach by following the ideas of Gerspach on determining the moments of weighted Steinhaus sums \cite{gerspach20}. One of the new inputs is from an application of the Gaussian random walk results seen in \cite{aharper201} to bound one of the integrals in a sharper manner (similar ideas can be seen in \cite{arguin2024fyodorov}). This is why we can get close to the exponents obtained by Caich and Mastrostefano. Most of the heavy probability are treated as a blackbox in this proof. Note that we only need two sets of test points to prove our result since we can afford less precision on bounding the Rademacher Euler product compared to the work of Hardy.\par
With the results we use to prove Theorem \ref{ubtheorem}, it is not very difficult to return to Theorem \ref{sharp_restricted_ub}. This comes down to mainly cosmetic tweaks to the work of Mastrostefano. The saving compared to Theorem \ref{ubtheorem} comes from only having to deal with one range of prime factors, which saves an application of the union bound (and generally makes the proof much simpler since one only has to consider two sets of test points rather than the three we use for the main result).\par
As for the lower bound, this is a small adaptation from Hardy's lower bound for $M_{f_{st}}(x)$ due to the differences between the Rademacher and Steinhaus Euler products. The Rademacher Euler product does not enjoy the translation invariance of the Steinhaus one, so is far more fiddly to interact with. However, the main idea of the proof remains the same: find a way to translate between $|M_f(t)|^2$ and an integral involving its Euler product without incurring too much loss, and continue our analysis in this setting. After several manipulations, one including Jensen's inequality, we find that there is a deterministic contribution in one of the lower bounds for the Euler product of size $\log\log(x)$ (there are no non-random contributions found in the Steinhaus setting by contrast). We then upper bound the random contribution, and conclude from there.
\subsection{Organisation of the paper and notation}
The proof of the upper bound takes the majority of the length of the paper and involves many reductions, so we detail them here. We begin proving the upper bound in Section \ref{upper_beginning}, where we proceed with a rigourous version of the splitting argument highlighted in equation (\ref{prime_split}) and handle the contribution of the integers which are $y_{j*}$ smooth. This yields five ""bad"" events and one ""good"" event that we have to manipulate. Four of the ""bad"" events are handled in Section \ref{easy_bounds}. The final ""bad"" event is by far the most complicated to handle, and this is the topic of Sections \ref{Q_1_part1} and \ref{ce} (since there are further subevents to deal with). Finally we handle the one ""good"" event and finish the proof of Theorem \ref{ubtheorem} in Section \ref{good_good}. We complete the proof of Theorem \ref{lbtheorem} in Section \ref{lower_bound_all} and finally the proof of Theorem \ref{sharp_restricted_ub} in Section \ref{mastro_variant}. We collect all of the results we need in Section \ref{prelims}. The limitations of the methods used will be briefly discussed at the end of Sections \ref{Q_good} and \ref{setup_heur}. From this point, we will always refer to the $k$-fold iterated logarithm as $\log_k(x)$ for tidiness and employ Vinogradov notation throughout the paper. We will use $P(n)$ to denote the largest prime which divides $n$ as well.
\subsection{Acknowledgements}
The author would like to thank his supervisors Joseph Najnudel and Oleksiy Klurman for carefully reading through previous versions of the paper. Maxim Gerspach, Seth Hardy, Rachid Caich and Besfort Shala helped with various
interesting discussions, suggestions and encouragement relating to this problem. The
author thanks Ofir Gorodetsky and Adam Harper for identifying errors in a previous
version of the paper and for their helpful comments. This work was supported by the
Heilbronn Institute for Mathematical Research.
\section{Preliminary results} \label{prelims}
For this upper bound, we will use much of the notation and framework established in Caich's proof of the unweighted Rademacher case. The first of these is the rough Rademacher hypercontractive inequality. This was first proved by Bonami \cite{bonami70} in a far more general setting than we present here, and reproved by Hal\'{a}sz \cite{halasz83}.
\begin{proposition}
    Let $k \in \mathbb{N}$ and $f$ be a Rademacher multiplicative function. Then for $(a_n) \subset \mathbb{C}$ where $a_n \neq 0$ for only finitely many $n$, then
    \begin{equation*}
        \mathbb{E}\left|\sum_{n}a_nf(n)\right|^{2k} \leqslant \left(\sum_n |a_n|^2\tau_{2k-1}(n)\right)^k 
    \end{equation*}
    where $\tau_{2q-1}(n) = \#\{(m_1,...,m_{2q-1}): m_1...m_{2q-1} = n\}$ is the $(2q-1)$-fold iterated divisor function. \label{prop_hyper}
\end{proposition}
The key number theory result we use is a Parseval identity for Dirichlet series, allowing us to move between partial sums of coefficients and their corresponding Euler product.
\begin{proposition} (Equation (5.26) in Section 5.1 of \cite{montgomery07}) Let $(a_n)_{n=1}^\infty$ be a sequence of complex numbers and $A(s) = \sum_{n=1}^\infty \frac{a_n}{n^s}$ be its associated Dirichlet series, with associated abscissa of convergence $\sigma_c$. Then for any $\sigma > \max\{\sigma_c,0\}$, we have that
\begin{equation*}
    \int_0^\infty \frac{\left|\sum_{n \leqslant x} a_n\right|^2}{x^{1+2\sigma}} dx = \frac{1}{2\pi}\int_{-\infty}^\infty \left|\frac{A(\sigma+it)}{\sigma+it}\right|^2 dt. 
\end{equation*} \label{HA1}
\end{proposition} 
This has been used many times in random multiplicative function literature, most strikingly in Harper's work proving Helson's conjecture \cite{aharper201} (among other results). We will use this slightly differently to Caich's application since we need to use different techniques to show that the various events coming from bounding these products almost surely occur. The approach used is similar to that of Gerspach's and Hardy's papers. We also want to exploit the martingale structure in $M_f(x)$. To do this, we use the same machinery as Caich and Mastrostefano.
\begin{proposition}
    (Doob's inequality for supermartingales, Theorem 9.1 in \cite{gut06}) Let $\lambda>0$ and suppose that the real sequence of random variables and $\sigma$-algebras $\{(X_n,\mathcal{F}_n)\}_{n \geqslant 0}$ is a non-negative supermartingale. Then
    $$\Prob\left(\max_{0 \leqslant n \leqslant N}X_n >\lambda\right) \leqslant \frac{\E(X_0)}{\lambda}.$$ \label{Doob_super}
\end{proposition}
\begin{proposition}
    (Doob's $L^r$ inequality, Theorem 9.4 in \cite{gut06}) Let $r > 1$. Suppose the real sequence of random variables and $\sigma$-algebras $\{(X_n,\mathcal{F}_n)\}_{0 \leqslant n \leqslant N}$ is a non-negative submartingale which is bounded in $L^r$. Then
    $$\E\left[\left(\max_{0 \leqslant n \leqslant N}X_n\right)^r\right] \leqslant \left(\frac{r}{r-1}\right)^r\max_{0 \leqslant n \leqslant N}\left(\E(X_n^r)\right).$$ \label{Lr-Doob}
\end{proposition}
To prove Theorem \ref{sharp_restricted_ub}, we will also need Doob's maximal inequality.
\begin{proposition}
    (Doob's maximal inequality, Theorem 9.1 in \cite{gut06}) Suppose $\lambda > 0$. Let $\{X_n,\mathcal{F}_n\}_{n \geqslant 0}$ be a sequence of random variables with associated $\sigma$-algebras which forms a non-negative submartingale. Then
    $$\Prob\left(\max_{0 \leqslant k \leqslant n}X_k > \lambda\right) \leqslant \frac{1}{\lambda}\E(|X_n|).$$ \label{Doob_max}
\end{proposition}
Finally, we give a key inequality used in Caich's work which is based on a version of the Azuma--Hoeffding inequality. We first state the standard version of said inequality, which is needed for the proof of Theorem \ref{sharp_restricted_ub}.
\begin{proposition}
    (Azuma--Hoeffding, \cite{hoeffding1994probability}) Let $(X_n)_{n \geqslant 0}: E \to \mathbb{R}$ be independent random variables that satisfy the bound $a_i < X_i(x) < b_i$ for any $x \in E$. Then we have
    $$\Prob\left(\sum_{k=1}^NX_k - \E\left(\sum_{k=1}^NX_k\right) > t\right) \leqslant 2\exp\left(-\frac{2t^2}{\sum_{k=1}^N(b_k-a_k)^2}\right).$$ \label{standard-AH}
\end{proposition}
\begin{proposition}
    (Caich, Lemma 3.12 in \cite{caich23}) Let $\{(X_n,\mathcal{F}_n\}_{n \leqslant N}$ be a complex sequence of martingale differences. Assume that $X_n$ is bounded almost surely (we suppose there exists real number $b>0$ such that $|X_n| < b$ almost surely). Further, assume that $|X_n|\leqslant S_n$ where $(S_n)_{n \leqslant N}$ is a sequence of real random variables, where for each $n$, $S_n$ is $\mathcal{F}_{n-1}$ measurable. Define the event $\mathcal{E} := \{\sum_{1 \leqslant n \leqslant N}S_n^2 \leqslant T\}$ where $T>0$ is a deterministic constant. Then for any $\epsilon>0$,
    $$\Prob\left(\left\{\left|\sum_{n \leqslant N}X_n\right| > \epsilon\right\} \bigcap \{\mathcal{E}\}\right) \leqslant 2\exp\left(-\frac{\epsilon^2}{10T}\right).$$ \label{Caich_AH_in}
\end{proposition}
This is proved in a similar way to a result of Pinelis, which is Theorem 3 in \cite{pinelis1992approach}. A significant part of the Euler product analysis will come from splitting the Euler products up into contributions from smaller and larger primes (in particular we hope to get more cancellation from the larger primes, and show that we can control the smaller primes in some sense). For this, we need a mildly modified version of one of Harper's Euler product results for Rademacher Euler products.
\begin{proposition}
    (Euler Product Result 2, \cite{harper202}) Suppose $400 \leqslant y < z$ for sufficiently large $z$ with $\alpha,\beta,t_1,t_2 \in \mathbb{R}$ where $\alpha,\beta$ are fixed and $\sigma \geqslant -1/\log(z)$, then \begin{align*}
        \E & \left(\prod_{y < p \leqslant z}\left|1+\frac{f(p)}{p^{1/2+\sigma+it_1}}\right|^{2\alpha}\left|1+\frac{f(p)}{p^{1/2+\sigma+it_2}}\right|^{2\beta}\right) \\
        = \exp\biggl\{&\sum_{y < p \leqslant z}\frac{\alpha^2 + \beta^2 + (\alpha^2 - \alpha)\cos(2t_1\log(p)) + (\beta^2-\beta)\cos(2t_2\log(p))}{p^{1+2\sigma}} \\
        &+ \sum_{y < p \leqslant z} \frac{2\alpha\beta(\cos((t_1-t_2)\log(p)) + \cos((t_1+t_2)\log(p))}{p^{1+2\sigma}} + O\left(\frac{1}{\sqrt{y}\log(y)}\right)\biggl\}.
    \end{align*} 
    For $\sigma \leqslant \frac{1}{\log(z)}$, then the above equals
    \begin{align*}
        =& \exp\left(O(\max\{\alpha,\beta,\alpha^2,\beta^2\}\left(1+\frac{|t_1|+|t_2|}{(\log(y))^{100}}\right)\right) \\ \cdot&\left(1+\min\left\{\frac{\log(z)}{\log(y)},\frac{1}{|t_1|\log(y)}\right\}\right)^{\alpha^2-\alpha} \left(1+\min\left\{\frac{\log(z)}{\log(y)},\frac{1}{|t_2|\log(y)}\right\}\right)^{\beta^2-\beta} \\
        \cdot& \left(\frac{\log(z)}{\log(y)}\right)^{\alpha^2+\beta^2}\left(\left(1+\min\left\{\frac{\log(z)}{\log(y)},\frac{1}{|t_1+t_2|\log(y)}\right\}\right)\left(1+\min\left\{\frac{\log(z)}{\log(y)},\frac{1}{|t_1-t_2|\log(y)}\right\}\right)\right)^{2\alpha\beta}.
    \end{align*} \label{euler_prod_result}
\end{proposition}
While Harper's result restricts to $\alpha,\beta \geqslant 0$, it does not change the method of proof (this is noted in the work of Gerspach in the proof of his Lemma 8 in \cite{gerspach20}, which is very similar to Euler Product Result 1 in \cite{harper202}). In practice, we are going to focus on $t_2 = 0$, and $-1/2 \leqslant t_1 \leqslant 1/2$, so the leading exponential term at the beginning of the second inequality can be safely ignored as a multiplicative error. We further specialise to the case $\alpha =1, \beta=-1$. Then for $t_2=0$, we have from the first statement in Proposition \ref{euler_prod_result},
$$\E\left(\prod_{y < p \leqslant z}\left|\frac{1+\frac{f(p)}{p^{1/2+\sigma+it_1}}}{1+\frac{f(p)}{p^{1/2+\sigma}}}\right|^{2}\right) \ll \exp\left(\sum_{y \leqslant p \leqslant z}\frac{4-4\cos(t_1\log(p))}{p^{1+2\sigma}}\right).$$
We then can Taylor expand the cosine in the above and use the inequality $1 - \cos(t) \leqslant \frac{t^2}{2}$ to show the above can be upper bounded by
$$\exp\left(\sum_{y < p \leqslant z}\frac{2t^2(\log(p))^2}{p^{1+2\sigma}}\right).$$
Choosing $y = 3/2$ and $z = [\exp(1/|t|)]$, we get for any $t \in \mathbb{R}$
\begin{equation}
    \E\left(\prod_{y < p \leqslant z}\left|\frac{1+\frac{f(p)}{p^{1/2+\sigma+it}}}{1+\frac{f(p)}{p^{1/2+\sigma}}}\right|^{2}\right) \ll \exp\left(Ct^2\frac{1}{t^2}\right) \ll 1 \label{short_euler_cancel}
\end{equation}
given $\sum_{p \leqslant z}\frac{(\log(p))^2}{p^{1+2\sigma}} \ll (\log(z))^2$ for this range of $\sigma$. This means we instead look at the short Euler product $|F_{e^{1/|t|}}(1/2+\sigma)|^2$ to consider the contribution from the "small" primes (the primes of size less than $\exp(1/|t|)$) as opposed to the shifted Euler product in $t$. This generates a significant saving for what we are looking at, since the expectation bound given by Proposition \ref{euler_prod_result} is rather wasteful when looking at $|F_{e^{1/|t|}}(1/2+\sigma)|^2$. This saving cannot be replicated when investigating the moments of the weighted Rademacher sums. \par
Another saving we obtain is using a sharp expectation bound on the mass of the Rademacher random Euler product over the interval $[N-1/2,N+1/2]$. This is not required in Hardy's work, since the main contribution to $M_f(x)$ for Steinhaus $f$ is concentrated on the near $t=0$, which is different in our case. We use this result to bound the contribution from Euler products on the range $[1/2,\log(x)]$. Caich also uses this result, but can apply it more directly than we can due to our weighted sums.
\begin{proposition}
    (Multiplicative Chaos Result 3, \cite{harper23}) Let $f(n)$ be a Rademacher multiplicative function. Then uniformly for sufficiently large $X$, any $q \in [0,1]$, $\sigma  \in [-1/\log(X),1/(\log(X))^{0.01}]$ and $|N| \leqslant (\log(X))^{1000}$, we have
    \begin{equation*}
        \E(\left(\int_{N-1/2}^{N+1/2}|F_X(1/2+\sigma+it)|^2 dt\right)^q) \ll (\log_2(|N|+10))^q\left(\frac{\min\{\log(X),1/|\sigma|\}}{1+(1-q)\sqrt{\log_2(X)}}\right)^q
    \end{equation*} \label{mult_chaos}
\end{proposition}
The proof of this result can be extrapolated from Key Proposition $3$ in Harper's paper on low moments of sums of random multiplicative functions \cite{aharper201}. For $q \in [0,1)$, then one can prove this type of result leaning more heavily into standard Gaussian multiplicative chaos techniques; see the recent works of Gorodetsky and Wong for example \cite{gorodetsky2024short}, \cite{gorodetsky2024martingale} and \cite{gorodetsky2025limiting}. This is not the only result we need to use from \cite{aharper201}, but his Proposition 6 (Key Proposition \ref{gaussian_walk} in this paper) requires a lot more work to state, so we delay this until Section \ref{ce}. We will provide the probability result used in this section (even though we do not use it in this form).
\begin{proposition}
    (Probability Result 1, \cite{aharper201}) Let $a\geqslant 1$. For any sufficiently large integer $n > 1$, let $(G_k)_{k=1}^n$ be a sequence of independent real Gaussian random variables with mean $0$ and variance between $1/20$ and $20$, say. Suppose $h$ is a function such that $|h(j)| < 10\log(j)$. Then
    $$\Prob\left(\sum_{m=1}^j G_m \leqslant a + h(j),  \ \forall \ 1 \leqslant j \leqslant n \right) \asymp \min\left\{\frac{a}{\sqrt{n}},1\right\}.$$ \label{ballot}
\end{proposition}
To deal with the random Euler products on the real line, we need an upper bound which Hardy uses in Section 2.7 of his paper.
\begin{proposition}
    (Upper exponential bound, Lemma 8.2.1 of Gut \cite{gut06}) Let $(X_n)_{n=1}^N$ be independent, mean $0$ random variables. Suppose $\sigma_k^2 = \mathbf{V}(X_k)$, $s_m^2 = \sum_{k \leqslant m}\sigma_k^2$ and that there exists $c_N > 0$ such that
    $$|X_m| \leqslant c_N s_N \text{ for } m=1,...,N$$
    Then for $x \in (0,1/c_N)$,
    $$\Prob\left(\sum_{k=1}^NX_k > xs_N\right) \leqslant \exp\left(-\frac{x^2}{2}\left(1-\frac{xc_N}{2}\right)\right).$$ \label{chernoff}
\end{proposition}
This is also known as a Chernoff bound in other probability contexts. \par

\section{Upper bound} \label{upper_beginning}
\subsection{Reduction to test points}
This proof combines Caich's work \cite{caich23} with a detailed analysis on a slightly different random Euler product following the approach of Gerspach \cite{gerspach20} and Hardy \cite{shardy23}. The first step is to rewrite the sum $\sum_{n \leqslant x}\frac{f(n)}{\sqrt{n}}$ where we split the sum according to the size of the prime factors of $n$. This was used in the work of Lau--Tenenbaum--Wu \cite{WTL13} to improve the efficiency of the inequality in Lemma \ref{prop_hyper}. The first appearance of this argument was in the work of Hal\'{a}sz on random multiplicative functions \cite{halasz83}.
\begin{lemma}
    (Lemma 2.4, \cite{WTL13}) Let $f$ be a Rademacher multiplicative function, and let $A>0$ be some fixed constant. Then there exists $\gamma := \gamma(A) \in (0,1)$ such that for $[y]$ denotes the integer part of $y$
    \begin{equation}
        x_i := [e^{i^{\gamma}}] \quad (i \geqslant 1) \label{point_def}
    \end{equation}
    we have that almost surely
    \begin{equation}
        \max_{x_{i-1} \leqslant x \leqslant x_i}\left|\sum_{x_{i-1} < n \leqslant x}f(n) \right| \ll \frac{\sqrt{x_i}}{(\log(x_{i}))^A}, \quad (i \geqslant 1). \label{test_point}
    \end{equation} \label{test_lem}
\end{lemma}
From this, we simply need to apply partial summation to show that
\begin{equation}
    \max_{x \in [x_{i-1}, x_i]}\left|\sum_{x_{i-1} < n \leqslant x}\frac{f(n)}{\sqrt{n}} \right| \ll \frac{\sqrt{x_{i}}}{\sqrt{x_{i-1}}}\frac{1}{\log(x_{i-1})} \ll \frac{1}{\log(x_{i-1})}. \label{test_bound}
\end{equation}
Note that in the result of Lau--Tenenbaum--Wu, their proof of Lemma \ref{test_lem} allows for an explicit choice of $\gamma$. We will find that in order to have 
$$\max_{x \in [x_{i-1}, x_i]}\left|\sum_{x_{i-1} < n \leqslant x}\frac{f(n)}{\sqrt{n}} \right| \ll \frac{1}{\log(x_i)},$$
one can choose any $\gamma < 1/320$. For the rest of this paper, we will choose $\gamma$ far smaller than this for reasons which will become more apparent throughout the proof (for our purposes, we assume $\gamma \leqslant 10^{-3}$). The bound obtained in equation (\ref{test_bound}) is stronger than we actually need, and this allows us to analyse only at the test points $x_i$ to deduce our upper bound for $M_f(x)$.
\begin{definition}
    \textit{Let $\epsilon>0$. Then we almost surely have for the $(x_i)$ defined in equation (\ref{point_def})}
    \begin{equation}
        M_f(x_{i}) \ll (\log_2(x_i))^{3/4+\epsilon}. \label{key_in}
    \end{equation} \label{key_prop}
\end{definition}
The combination of equations (\ref{test_bound}) and (\ref{key_in}) is sufficient to prove the upper bound, which we will complete after proving the proposition above. \par
\subsection{Splitting $M_f(x_i)$ by prime factors}
To prove Key Proposition \ref{key_prop}, we need to use some of definitions used in Caich's work and before that Mastrostefano. For $\ell \in \mathbb{N}$, we take $X_{\ell} = \exp\left(2^{\ell^K}\right)$, where $K=\left\lfloor\frac{25}{\epsilon}\right\rfloor$. In particular, we see that each $x_i$ is contained in the interval $(X_{\ell-1},X_{\ell}]$ for some $\ell \in \mathbb{N}$. As observed in both Caich and Mastrostefano's work, this makes our task more difficult than what we see in Hardy and Lau--Tenenbaum--Wu since we will have more test points $x_i$ in the interval $(X_{\ell-1},X_{\ell}]$. This is why we need the more powerful martingale machinery to assist us. We also define the finite sequence of numbers $(y_j)_{j=0}^J$ where we have that
\begin{equation*}
    y_0 = \exp\left(2^{\ell^K(1-K/\ell)}\right), \quad y_j = \exp\left(e^{j/\ell}2^{\ell^K(1-K/\ell)}\right).
\end{equation*}
where $J$ is the smallest integer such that $y_J > X_i$. Then we have that
\begin{equation*}
    J \ll \ell^K \asymp \log_2(x_i) \asymp \log_2(y_j),
\end{equation*}
for $x_i \in (X_{\ell-1},X_{\ell}]$ and $1 \leqslant j \leqslant J$ when $\ell$ is sufficiently large. At this point, we also let $j^*$ be the largest $j$ such that
\begin{equation}
    \frac{\log(X_{\ell-1})}{{\log(y_j)}} > \ell^{2K/\gamma}. \label{j^*_def}
\end{equation}
We choose this cutoff because it ensures that each $x_i$ has a large enough proportion of the sum covered by the martingale structure. The number of $j^* < j \leqslant J$ is $\ll \ell^K$. This will be important when it comes to bounding the first contribution over the very smooth numbers. A crucial point is that $j^*$ only depends on $\ell$. For future improvements, one might consider using prime ranges which depend on $i$ instead.
Then we formulate a sequence of events $\mathcal{A}_{\ell}$, with the idea being to show that the probability of $\mathcal{A}_{\ell}$ occuring is summable in $\ell$ and then finally applying the first Borel--Cantelli lemma to obtain that the event $\mathcal{A}_{\ell}$ occurs at most finitely often, which is sufficient to prove Proposition \ref{key_prop}. We define
\begin{equation}
    \mathcal{A}_{\ell} := \left\{\sup_{X_{\ell-1} < x_i \leqslant X_{\ell}} \frac{|M_f(x_i)|}{(\log_2(x_i))^{3/4+\epsilon}} > 4\right\} \label{A_def}
\end{equation}
We want to partition the event $\mathcal{A}_\ell$ according to the following decomposition of $M_f(x_i)$.
\begin{align*}
    M_f(x_i) =& S_{i,0} + \sum_{1 \leqslant j \leqslant J}S_{i,j}; \\
    S_{i,0} =& \sum_{\substack{n \leqslant x_i\\P(n) \leqslant y_0}}\frac{f(n)}{\sqrt{n}}; \\
    S_{i,j} =& \sum_{y_{j-1} < p \leqslant y_j}\frac{f(p)}{\sqrt{p}} \sum_{\substack{n \leqslant x_i/p\\P(n) < p}}\frac{f(n)}{\sqrt{n}}.
\end{align*}
We distinguish the two cases of $y_j$ on their size compared to $x_i$, where we have that $\frac{\log(X_\ell)}{\log(y_{j})} > \ell^{2K/\gamma}$ and $\frac{\log(X_\ell)}{\log(y_{j})} \leqslant \ell^{2K/\gamma}$. This allows us to define two new events:
\begin{align}
    \mathcal{B}_{0,\ell} :=& \left\{ \sup_{X_{\ell-1} < x_i \leqslant X_\ell}\sum_{0 \leqslant j \leqslant j^*}\frac{|S_{i,j}|}{(\log_2(x_i))^{1/4+\epsilon}} > 2\right\} ;\\
    \mathcal{B}_{1,\ell} :=& \left\{ \sup_{X_{\ell-1} < x_i \leqslant X_\ell}\sum_{j^* < j \leqslant J}\frac{|S_{i,j}|}{(\log_2(x_i))^{3/4+\epsilon}} > 2\right\}. \label{B1_def}
\end{align}
By the triangle inequality, we see that $\Prob(\mathcal{A}_\ell) \leqslant \Prob(\mathcal{B}_{0,\ell}) + \Prob(\mathcal{B}_{1,\ell})$, so to prove Key Proposition \ref{key_prop}, it is sufficient to show that $\Prob(\mathcal{B}_{0,\ell})$ and $\Prob(\mathcal{B}_{1,\ell})$ are summable in $\ell$. Note that we introduced the splitting on the size of the primes much earlier than Caich does. This is because we have to proceed differently with the smaller primes since we do not get as much cancellation from considering only smooth numbers. We note that there are at most $\frac{2K\ell\log(\ell)}{\gamma} \ll \ell\log(\ell)$ such $y_j$ of the larger primes. We now prove $\Prob(\mathcal{B}_{0,\ell})$ is summable conditional on the complement of the event $\mathcal{H}_\ell'$,
\begin{equation}
    \mathcal{H}_\ell ' := \left\{\left|\frac{F_{y_{j^*}}(1/2)}{(\log(y_{j^*}))^{-1/10}}\right| > A\right\},\label{G_prime}
\end{equation}
where $A>0$ is a large constant.
\begin{lemma}
    $\Prob(\mathcal{B}_{0,\ell})$ is summable in $\ell$ given $\Prob(\mathcal{H}_\ell')$ is summable. \label{PB_0k}
\end{lemma}
\begin{proof}
    We can rewrite the sum in the event $\mathcal{B}_{0,\ell}$ in the following way;
    $$\sum_{0 \leqslant j \leqslant j^*}\frac{f(n)}{\sqrt{n}} = F_{y_{j^*}}(1/2) - \sum_{\substack{n > x_i \\ P(n) \leqslant y_{j^*}}}\frac{f(n)}{\sqrt{n}},$$
    since $\sum_{P(n) < y_{j^*}}\frac{f(n)}{\sqrt{n}} = F_{y_{j^*}}(1/2)$.
    We do not evaluate the expectation of the sum directly for the same reasons given in Gerspach \cite{gerspach20} since a direct application of Chebychev's inequality is too wasteful. This is why we introduced the additional splitting in the sum so early. Then, we use the triangle inequality and find
    \begin{align*}
        \Prob(\mathcal{B}_{0,\ell}) \leqslant& \ \Prob\left(\frac{|F_{y_{j^*}}(1/2)|}{(\log_2(X_{\ell-1}))^{1/4+\epsilon}} > 1\right) \\
        +& \ \Prob\left(\sup_{X_{\ell-1} < x_i \leqslant X_\ell}\frac{\left|\sum_{\substack{n > x_i \\ P(n) \leqslant y_{j^*}}}\frac{f(n)}{\sqrt{n}}\right|}{(\log_2(x_i))^{1/4+\epsilon}} > 1\right).
    \end{align*}
    The first probability can be bounded above by $\Prob(\mathcal{H}_\ell')$ (we are multiplying by a large factor in $\mathcal{H}_\ell'$ whereas we are dividing by a factor greater than $1$ here), so we only need to look at the second sum. From the union bound, Chebychev's inequality and Proposition \ref{prop_hyper}, we have
    $$\Prob(\mathcal{B}_{0,\ell}) \leqslant \sum_{X_{\ell-1} < x_i \leqslant X_\ell} \frac{\sum_{\substack{n > x_i \\ P(n) \leqslant y_{j^*}}}\frac{1}{n}}{(\log_2(x_i))^{1/2+2\epsilon}} + \Prob(\mathcal{H}_\ell').$$
    We assumed the second sum converges, so it is sufficient to understand the first sum. This can be upper bounded using Rankin's trick,
    $$\sum_{\substack{n > x_i \\ P(n) \leqslant y_{j^*}}}\frac{1}{n} \leqslant x_i^{-1/\log(y_{j^*})}\prod_{p \leqslant y_{j^*}}\left(1-\frac{1}{p^{1-\frac{1}{\log(y_{j^*})}}}\right)^{-1} \ll \frac{\log(y_{j^*})}{x_i^{1/\log(y_{j^*})}}.$$
    From the definition of $j^*$, we have
    $$x_i^{1/\log(y_{j^*})} > \exp(\ell^{2K/\gamma}).$$
    From this, we see
    $$\Prob(\mathcal{B}_{0,\ell}) \ll \sum_{X_{\ell-1} < x_i \leqslant X_\ell}\frac{\log(y_{j^{*}})}{x_i^{1/\log(y_{j^*})}(\log_2(x_i))^{1/2+2\epsilon}} + \Prob(\mathcal{H}_\ell') \ll \frac{(\log(X_{\ell}))^{1/\gamma+1}}{\exp(\ell^{2K/\gamma})}+\Prob(\mathcal{H}_\ell'),$$
    which is summable in $\ell$ using that $\log(X_\ell) = 2^{\ell^K}$ (note that we have ignored the $(\log_2(x_i))^{1/2+2\epsilon}$ in the denominator. This is because it is insignificant in size to the other terms, and for the next remark).
\end{proof}
One could take this further to combine a version of Lemma \ref{PB_0k} (with a stronger barrier) with Lemma \ref{test_lem} to show that
$$\sum_{\substack{n \leqslant x \\ P(n) \leqslant y_{j^*}}}\frac{f(n)}{\sqrt{n}}$$
converges for example. An immediate question would be to find how large can the smoothness parameter in terms of $x$ before the sum diverges.  In particular, it would be of interest to sharpen the result of Lau--Tenenbaum--Wu in Lemma \ref{test_lem} in order to attain a smaller power of $\log_2(x)$ in the denominator. Significant improvements would come by choosing a less sparse set of test points $X_\ell$ for example (which we choose not to do here).
\subsection{Decomposing the $S_{i,j}$} \label{sum_decomp}
This section is where we manipulate the decomposition into several subquantites to bound individually. We follow Caich and Mastrostefano. The main improvements we gain over Hardy's approach is that we are considering the quantity
$$ S_{i,j} := \sum_{y_{j-1} < p \leqslant y_j}\frac{f(p)}{\sqrt{p}}\sum_{\substack{n \leqslant x_i/p\\P(n) < p}}\frac{f(n)}{\sqrt{n}}$$
which enjoys some extra martingale structure compared to his $S_{i,j}$ (by splitting on individual primes as opposed to splitting only on ranges of primes). We can do this because the Rademacher random multiplicative functions are supported only on squarefree integers, so there is always a unique largest prime factor to factor out using the multiplicativity of the $f(p)$.\par
Given that $f$ is a Rademacher multiplicative function and $(\mathcal{F}_p)_p $ is the $\sigma$-algebra generated by the random variables $f(q)$ where $q < p$, then we have that
$$\E\left(f(p)\sum_{\substack{n \leqslant x_i/p \\ P(n) < p}}\frac{f(n)}{\sqrt{n}} \bigg| \mathcal{F}_{p}\right) = 0.$$
In particular, this shows that $S_{i,j}$ is a sum of martingale differences with variance
\begin{equation*}
    V_\ell(x_i,y_j;f) := \sum_{y_{j-1} < p \leqslant y_j}\frac{1}{p}\left|\sum_{\substack{n \leqslant x_i/p \\ P(n) < p}}\frac{f(n)}{\sqrt{n}}\right|^2.
\end{equation*}
Applying Proposition \ref{Caich_AH_in} reduces bounding the $S_{i,j}$ to understanding $V_\ell(x_i,y_j)$. This is beneficial since it allows us to take a much lower moment than what is taken in Lau--Tenenbaum--Wu and Hardy's work for example. What also helps us improve in the accuracy of this section is to treat the different ranges of $j$ separately as opposed to dealing with them simultaneously. The sparser choice of the $X_\ell$ will allow us to is will also allow us to take advantage of Harper's better than squareroot cancellation results to a fuller extent as well using this sparser choice of $X_\ell$. Throughout this section we have the restriction on $j$ such that $j^* < j \leqslant J$ since we have already handled the case with the smaller $y_j$.\par
We let $\mathcal{X}$ be some large real number which will be chosen later (we will choose $\log(\mathcal{X}) \asymp \ell^K$ in a future section), $p$ be a prime with $p < t < p(1+1/\mathcal{X})$. Then we can upper bound $V_\ell(x_i,y_j;f)$ using the triangle inequality
$$V_\ell(x_i,y_j;f) \leqslant 2\mathcal{C}_\ell(x_i,y_j;f) + 2\mathcal{D}_\ell(x_i,y_j;f)$$
where 
\begin{align}
    \mathcal{C}_\ell(x_i,y_j;f) :=& \sum_{y_{j-1} < p \leqslant y_j}\frac{\mathcal{X}}{p^2}\int_{p}^{p(1+1/\mathcal{X})}\left|\sum_{\substack{n \leqslant x_i/t \\ P(n) <p}}\frac{f(n)}{\sqrt{n}}\right|^2 dt ;\label{C_def} \\
    \mathcal{D}_\ell(x_i,y_j;f) :=& \sum_{y_{j-1} < p \leqslant y_j}\frac{\mathcal{X}}{p^2}\int_{p}^{p(1+1/\mathcal{X})}\left|\sum_{\substack{x_i/t <n \leqslant x_i/p\\P(n)<p}}\frac{f(n)}{\sqrt{n}}\right|^2 dt. \label{D_def}
\end{align}
The main contribution will come from the $\mathcal{C}_\ell(x_i,y_j;f)$ term, so we begin there. Following Caich's decomposition, we have after applying Fubini--Tonelli
$$\mathcal{C}_\ell(x_i,y_j;f) \ll \mathcal{C}_\ell^{(1)}(x_i,y_j;f) + \mathcal{C}_\ell^{(2)}(x_i,y_j;f)$$
where
\begin{align}
    \mathcal{C}^{(1)}_\ell(x_i,y_j;f) :=& x_i\int_{x_i/y_{j}}^{x_i/y_{j-1}}\sum_{\max\left\{\frac{x_i}{z(1+1/\mathcal{X})},y_{j-1}\right\} < p \leqslant \frac{x_i}{z}}\frac{\mathcal{X}}{p^2}\left|\sum_{\substack{n \leqslant z\\P(n) < p}}\frac{f(n)}{\sqrt{n}}\right|^2 \frac{dz}{z^2} \label{C_1_def};\\
    \mathcal{C}_{\ell}^{(2)}(x_i,y_j;f) :=& x_i\int_{\frac{x_i}{y_j(1+1/\mathcal{X})}}^{x_i/y_j}\sum_{\max\left\{\frac{x_i}{z(1+1/\mathcal{X})},y_{j-1}\right\} < p \leqslant \min\left\{\frac{x_i}{z},y_j\right\}}\frac{\mathcal{X}}{p^2}\left|\sum_{\substack{n \leqslant z \\P(n) < p}}\frac{f(n)}{\sqrt{n}}\right|^2 \frac{dz}{z^2}. \label{C_2_def}
\end{align}
after we made the variable change $z=x_i/t$. Applying Mertens Theorem with Abel summation, we have that
\begin{equation}
    \sum_{\frac{x_i}{z(1+1/\mathcal{X})} < p \leqslant x_i/z}\frac{\mathcal{X}}{p^2} \ll \frac{z}{x_i\log(x_i/z)}. \label{p_bound}
\end{equation}
This can be further refined by observing $y_{j-1} \leqslant x_i/z \leqslant y_j$ and $\log(y_j) = e^{1/\ell}\log(y_{j-1})$, so $\log(x_i/z) \asymp \log(y_j)$. In particular, we have that
$$\mathcal{C}^{(1)}_\ell(x_i,y_j;f) \ll \frac{1}{\log(y_j)}\int_{x_i/y_{j}}^{x_i/y_{j-1}}\sup_{\frac{x_i}{z(1+1/\mathcal{X})} < q \leqslant \frac{x_i}{z}}\left|\sum_{\substack{n \leqslant z\\P(n) < q}}\frac{f(n)}{\sqrt{n}}\right|^2\frac{dz}{z}.$$
Following Caich, we are now in a position to define the following quantities:
\begin{align}
    \mathcal{Q}_\ell^{(1)}(x_i,y_j;f) :=& \frac{1}{\log(y_j)}\int_{x_i/y_{j}}^{x_i/y_{j-1}}\left|\sum_{\substack{n \leqslant z \\ P(n) < \frac{x_i}{z}}}\frac{f(n)}{\sqrt{n}}\right|^2 \frac{dz}{z} \label{Q_1}; \\
    \mathcal{Q}_\ell^{(2)}(x_i,y_j;f) :=& \frac{1}{\log(y_j)}\int_{x_i/y_{j}}^{x_i/y_{j-1}}\sup_{\frac{x_i}{z(1+1/\mathcal{X})} \leqslant q \leqslant \frac{x_i}{z}}\left|\sum_{\substack{n \leqslant z \\ \frac{x_i}{z(1+1/\mathcal{X})} \leqslant P(n)  < q}}\frac{f(n)}{\sqrt{n}}\right|^2 \frac{dz}{z} \label{Q_2}; \\
    \mathcal{Q}_\ell^{(3)}(x_i,y_j;f) :=& \frac{1}{\log(y_j)}\int_{x_i/y_{j}}^{x_i/y_{j-1}}\left|\sum_{\substack{n \leqslant z \\ \frac{x_i}{z(1+1/\mathcal{X})} \leqslant P(n) < \frac{x_i}{z}}}\frac{f(n)}{\sqrt{n}}\right|^2 \frac{dz}{z}. \label{Q_3}
\end{align}
\section{Bounding the complement events} \label{easy_bounds}
In this section we will show that various complement events relating to the quantities in the previous section are summable in $\ell$. Let $T(\ell) \geqslant \ell^{10}$. Then we define the events
\begin{align}
    \mathcal{D}_\ell :=& \ \left\{\sup_{X_{\ell-1} < x_i \leqslant X_\ell}\sum_{j^* < j \leqslant J}\mathcal{D}_\ell(x_i,y_j;f) > \frac{T(\ell)}{\ell^{K/2}}\right\}; \\
    \mathcal{C}_\ell^{(2)} :=& \ \left\{\sup_{X_{\ell-1} < x_i \leqslant X_\ell}\sum_{j^* < j \leqslant J}\mathcal{C}_\ell^{(2)}(x_i,y_j;f) > \frac{T(\ell)}{\ell^{K/2}}\right\}; \\
    \mathcal{Q}_\ell^{(1)} :=& \ \left\{\sup_{X_{\ell-1} < x_i \leqslant X_\ell} \sup_{j^* < j \leqslant J}\mathcal{Q}_\ell^{(1)}(x_i,y_j;f) > \frac{T(\ell)\ell^{K/2}}{\ell\log(\ell)}\right\}; \\
    \mathcal{Q}_\ell^{(2)} :=& \ \left\{\sup_{X_{\ell-1} < x_i \leqslant X_\ell} \sup_{j^* < j \leqslant J}\mathcal{Q}_\ell^{(2)}(x_i,y_j;f) > \frac{T(\ell)}{\ell^{K/2}\ell\log(\ell)}\right\}; \\
    \mathcal{Q}_\ell^{(3)} :=& \ \left\{\sup_{X_{\ell-1} < x_i \leqslant X_\ell} \sup_{j^* < j \leqslant J}\mathcal{Q}_\ell^{(3)}(x_i,y_j;f) > \frac{T(\ell)}{\ell^{K/2}\ell\log(\ell)}\right\}.
\end{align}
We will show that all of these events occur with a probability which is summable in $\ell$. Then we will apply the first Borel--Cantelli lemma to deduce that these events occur finitely many times. After this, we will almost be done with this proof, with the last thing to do being applying Proposition \ref{Caich_AH_in} to handle the good event.\par
We will leave the analysis of the $\mathcal{Q}_\ell^{(k)}$ for last since these will need some extra tools to deal with this event (with $\mathcal{Q}_\ell^{(1)}$ requiring the most work and demanding multiple sections to address since we split it up many times). All the complement events other than $\mathcal{Q}_\ell^{(1)}$ are sharp enough for a stronger upper bound to be proved, and there is only one facet of $\mathcal{Q}_\ell^{(1)}$ that needs to be improved in order to lower the upper bound to a smaller power of $\log_2(x)$. This will be discussed in more detail when we deal with that event. We start with bounding the probability of $\mathcal{D}_\ell$. All of these proofs are very similar to the ones given by Caich in his paper with a few tweaks.\par
First we show that $\mathcal{D}_\ell$ occurs only finitely many times (the reason for this order will become apparent when bounding the $\mathcal{C}_\ell^{(k)}$). 
\begin{lemma}
    The sum $\sum_{\ell \geqslant 1}\Prob(\mathcal{D}_\ell)$ converges.
    \label{D_conv}
\end{lemma}
\begin{proof}
Let $r \geqslant 1$ be a constant to be chosen later. Then we can apply the union bound and Markov's inequality to the $r$-th power followed by Minkowski's inequality to see
\begin{align*}
    \Prob(\mathcal{D}_\ell) \leqslant& \frac{1}{T(\ell)^r}\sum_{X_{\ell-1} < x_i \leqslant X_\ell}\left(\sum_{j^* < j \leqslant J}\ell^{K/2}\E\left[\mathcal{D}_\ell(x_i,y_j;f)\right]\right)^r \\
    \ll& \sum_{X_{\ell-1} < x_i \leqslant X_\ell}\ell^{Kr/2}\left(\sum_{1 < j \leqslant J}\left(\E\left[(\mathcal{D}_\ell(x_i,y_j;f))^r\right]\right)^{1/r}\right)^r.
\end{align*}
We next proceed with finding suitable upper bounds for $\left(\E\left[(\mathcal{D}_\ell(x_i,y_j;f))^r\right]\right)^{1/r}$. We follow the approach of Harper \cite{harper202} and Hardy \cite{shardy23}. We have
$$(\E\left[\mathcal{D}_\ell(x_i,y_j;f)^r\right])^{1/r} \leqslant \sum_{y_{j-1} < p \leqslant y_j}\frac{1}{p}\left[\E\left(\frac{\mathcal{X}}{p}\int_{p}^{p(1+1/\mathcal{X})}\left|\sum_{\substack{\frac{x_i}{t} < n \leqslant \frac{x_i}{p}\\P(n) <p}}\frac{f(n)}{\sqrt{n}}\right|^2 dt\right)^r\right]^{1/r}.$$
Since we have normalised the integral above, we can apply H\"{o}lder's inequality,
\begin{align*}
    (\E\left[\mathcal{D}_\ell(x_i,y_j;f)^r\right])^{1/r} \leqslant \sum_{y_{j-1} < p \leqslant y_j}\frac{1}{p}\left(\frac{\mathcal{X}}{p}\int_{p}^{p(1+1/\mathcal{X})}\E\left(\left|\sum_{\substack{\frac{x_i}{t} < n \leqslant \frac{x_i}{p} \\ P(n) <p}}\frac{f(n)}{\sqrt{n}}\right|^{2r}\right) dt\right)^{1/r}.
\end{align*}
We split the above integral according to $p > \frac{x_i}{1+\mathcal{X}}$. In the first case, we find that $\frac{x_i}{p}-\frac{x_i}{p(1+1/\mathcal{X})} < 1$, so the inner sum has at most one term. This gives that for this range
$$\frac{\mathcal{X}}{p}\int_{p}^{p(1+1/\mathcal{X})}\E\left(\left|\sum_{\substack{\frac{x_i}{t} < n \leqslant \frac{x_i}{p}\\P(n) < p}}\frac{f(n)}{\sqrt{n}}\right|^{2r}\right)dt \ll \frac{\mathcal{X}}{p}\int_{p}^{p(1+1/\mathcal{X})}\left(\frac{t}{x_i}\right)^{r}dt \ll \left(\frac{p}{x_i}\right)^r.$$
Summing over the $p>\frac{x_i}{1+\mathcal{X}}$ and using Abel summation with the prime number theorem we have
$$\frac{1}{x_i}\sum_{\frac{x_i}{1+\mathcal{X}} < p \leqslant x_i}1 \ll \frac{1}{\log(x_i)}.$$
For the range $p \leqslant \frac{x_i}{1+\mathcal{X}}$, then we need to proceed slightly differently. We follow Hardy's treatment of this problem \cite{shardy23}, where we use Proposition \ref{prop_hyper} followed by Cauchy--Schwarz. For the first step, we have
$$\E\left(\left|\sum_{\substack{\frac{x_i}{t} < n \leqslant \frac{x_i}{p}\\P(n) < p}}\frac{f(n)}{\sqrt{n}}\right|^{2r}\right) \leqslant \left(\sum_{\substack{\frac{x_i}{t} < n \leqslant \frac{x_i}{p}\\P(n) <p}}\frac{\tau_{2r-1}(n)}{n}\right)^{r}.$$
Applying Cauchy--Schwarz, we find
$$\left(\sum_{\substack{\frac{x_i}{t} < n \leqslant \frac{x_i}{p}\\P(n) <p}}\frac{\tau_{2r-1}(n)}{n}\right)^{r} \leqslant \left(\left(\sum_{\substack{\frac{x_i}{t} < n \leqslant \frac{x_i}{p}\\P(n)<p}}\frac{1}{n^2}\right)\left(\sum_{\substack{\frac{x_i}{t} < n \leqslant \frac{x_i}{p}\\P(n) < p}}\left(\tau_{2r-1}(n)\right)^2\right)\right)^{r/2}.$$
We have
$$\left(\sum_{\substack{\frac{x_i}{t} < n \leqslant \frac{x_i}{p}\\P(n) <p}}\frac{1}{n^2}\right)^{r/2} \ll \left(\sum_{\frac{x_i}{p(1+1/X)} < n \leqslant \frac{x_i}{p}}\frac{1}{n^2}\right)^{r/2} \ll \left(\frac{p}{\mathcal{X} x_i}\right)^{r/2}.$$
The second term we apply Cauchy--Schwarz to this sum and observe that $(\tau_{2r-1}(n))^{2} \leqslant \tau_{4r^2-4r+1}(n)$ (as seen in the paper of Benatar--Nishry--Rodgers \cite{benatar22}),
$$\left(\sum_{\substack{\frac{x_i}{t} < n \leqslant \frac{x_i}{p}\\P(n) < p}}(\tau_{2r-1}(n))^2\right)^{r/2} \leqslant \left(\sum_{\substack{\frac{x_i}{t} < n \leqslant \frac{x_i}{p}\\P(n) < p}}\tau_{4r^2-4r+1}(n)\right)^{r/2} \ll \left(\frac{x_i(\log(x_i))^{4r^2-4r}}{p}\right)^{r/2}.$$
Putting this all together, we have that for $p \leqslant \frac{x_i}{1+\mathcal{X}}$
$$\E\left(\left|\sum_{\substack{\frac{x_i}{t} < n \leqslant \frac{x_i}{p}\\P(n)<p}}\frac{f(n)}{\sqrt{n}}\right|^{2r}\right) \ll \left(\frac{(\log(x_i))^{4r^2-4r}}{\mathcal{X}}\right)^{r/2}.$$
This leads to us having
$$\left(\E\left[(\mathcal{D}_\ell(x_i,y_j;f)^r\right]\right)^{1/r} \ll \sum_{y_{j-1}<p \leqslant y_j}\frac{1}{p}\frac{(\log(x_i))^{2r^2-2r}}{\sqrt{\mathcal{X}}} + \frac{1}{\log(x_i)}.$$
This can be substituted into the initial bound on $\Prob(\mathcal{D}_\ell)$;
$$\Prob(\mathcal{D}_\ell) \ll \sum_{X_{\ell-1} < x_i \leqslant X_\ell}\left(\sum_{p \leqslant \frac{x_i}{1+\mathcal{X}}}\frac{1}{p}\frac{\ell^{K/2}(\log(x_i))^{2r^2-2r}}{\sqrt{\mathcal{X}}}\right)^{r} + \sum_{X_{\ell-1} < x_i \leqslant X_\ell}\left(\frac{\ell^{K/2}}{\log(x_i)}\right)^r.$$
We then choose $\mathcal{X} = (\log(x_i))^{4r^2-4r+2}$. With this choice, we have
$$\Prob(\mathcal{D}_\ell) \ll \sum_{X_{\ell-1} < x_i \leqslant X_\ell} (\log_2(x_i))^{r/2}\left(\frac{\ell^{K/2}}{\log(x_i)}\right)^r.$$
Finally, after choosing $r > 1/\gamma$ where $\gamma$ was chosen in Lemma \ref{test_lem}, we have that $\sum_{\ell \geqslant 1}\Prob(\mathcal{D}_\ell)$ is certainly summable in $\ell$.
\end{proof}
The next lemma is also rather easy to obtain using Markov's inequality and partial summation.
\begin{lemma}
    The sum $\sum_{\ell \geqslant 1}\Prob\left(\mathcal{C}_\ell^{(2)}\right)$ converges. \label{C_2conv}
\end{lemma}
\begin{proof}
We proceed by applying Markov's inequality and partial summation to equation \ref{p_bound} to find
$$\Prob\left(\mathcal{C}_\ell^{(2)}\right) \leqslant \frac{\ell^{K/2}}{T(\ell)}\sum_{X_{\ell-1} < x_i \leqslant X_\ell} \sum_{j^* < j \leqslant J}x_i \int_{\frac{x_i}{y_j(1+1/\mathcal{X})}}^{x_i/y_j} \sum_{\max\{y_{j-1},\frac{x_i}{z(1+1/\mathcal{X})}\} < p \leqslant \frac{x_i}{z}}\frac{\mathcal{X}\log(p)}{p^2}\frac{dz}{z^2},$$
using the bound
$$\sum_{\substack{n \leqslant x \\ P(n) \leqslant z}}\frac{1}{n} \leqslant \sum_{\substack{n \geqslant 1 \\P(n) \leqslant z}}\frac{1}{n} = \prod_{p \leqslant z}\left(1-\frac{1}{p}\right)^{-1} \ll \log(z).$$Similarly to equation (\ref{p_bound}), we have that since $\frac{x_i}{z(1+1/\mathcal{X})} \leqslant \max\{\frac{x_i}{z(1+1/\mathcal{X})},y_{j-1}\}$,
$$\Prob\left(\mathcal{C}_\ell^{(2)}\right) \ll \frac{\ell^{K/2}}{T(\ell)}\sum_{X_{\ell-1} < x_i \leqslant X_\ell}\sum_{j^* < j \leqslant J}\int_{\frac{x_i}{y_j(1+1/\mathcal{X})}}^{x_i/y_j}\frac{dz}{z}.$$
We find that
$$\Prob\left(\mathcal{C}_\ell^{(2)}\right) \ll \frac{\ell^{K/2}}{\mathcal{X}T(\ell)}\sum_{X_{\ell-1} < x_i \leqslant X_\ell} \sum_{j^* < j \leqslant J} 1.$$
Using our choice of $\mathcal{X}$ for $i$ sufficiently large when bounding $\Prob(\mathcal{D}_\ell)$, this is sufficient to prove that $\Prob\left(\mathcal{C}_\ell^{(2)}\right)$ is summable in $\ell$.
\end{proof}
This leaves the events $\mathcal{Q}_\ell^{(k)}$ to bound. These will require the martingale techniques mentioned before to deal with. We start with handling $\Prob\left(\mathcal{Q}_\ell^{(2)}\right)$ and $\Prob\left(\mathcal{Q}_\ell^{(3)}\right)$ since they both involve very similar methods. We next prove the following lemma.
\begin{lemma}
    For $z \geqslant 1$ and $q_0$ some positive integer, let
    \begin{equation}
        Y_{q_0,q}(z) : = \sum_{\substack{n \leqslant z\\q_0 \leqslant P(n) <q}}\frac{f(n)}{\sqrt{n}}. \label{Y_Def}
    \end{equation}
    Then $(|Y_{q_0,q}(z)|)_{q}$ is a submartingale with respect to filtration $(\mathcal{F}_q)$. \label{Y_sub}
\end{lemma}
\begin{proof}
Let $q < p$ be two consecutive prime numbers. Then we have that
\begin{align*}
    \E(|Y_{q_0,p}(z)| \ |\mathcal{F}_p) =& \E(\left|Y_{q_0,q}(z) + \frac{f(p)}{\sqrt{p}}Y_{q_0,q}(z/p)\right| \ |  \mathcal{F}_p) \\
    =& \frac{1}{2}\left(\left|Y_{q_0,q}(z) + \frac{Y_{q_0,q}(z/p)}{\sqrt{p}}\right| + \left|Y_{q_0,q}(z) - \frac{Y_{q_0,q}(z/p)}{\sqrt{p}}\right|\right) \\
    \geqslant& |Y_{q_0,q}(z)|.
\end{align*}
\end{proof}
\begin{lemma}
    The sums $\sum_{\ell \geqslant 1}\Prob\left(\mathcal{Q}_\ell^{(2)}\right)$ and $\sum_{\ell \geqslant 1}\Prob\left(\mathcal{Q}_\ell^{(3)}\right)$ converge. \label{Q_conv}
\end{lemma}
\begin{proof}
To show $\Prob\left(\mathcal{Q}_\ell^{(2)}\right)$ is summable in $\ell$, we apply Chebychev's inequality, followed by Cauchy--Schwarz and find
\begin{align*}
    \Prob\left(\mathcal{Q}_\ell^{(2)}\right) \leqslant& \frac{\ell^{K+2}(\log(\ell))^2}{(T(\ell))^2}\sum_{X_{\ell-1} < x_i \leqslant X_\ell}\sum_{j^* < j \leqslant J}\frac{1}{(\log(y_j))^2}\left(\int_{x_i/y_j}^{x_i/y_{j-1}}\frac{dz}{z}\right) \\
    \times& \int_{x_i/y_j}^{x_i/y_{j-1}}\E\left(\sup_{\frac{x_i}{z(1+1/\mathcal{X})} \leqslant q \leqslant \frac{x_i}{z}}\left|Y_{\frac{x_i}{z(1+1/\mathcal{X})},q}(z)\right|^4\right) \frac{dz}{z}.
\end{align*}
We can now apply Doob's $L^4$ inequality (Proposition \ref{Lr-Doob} where $r=4)$, which leaves us with
$$\Prob\left(\mathcal{Q}_\ell^{(2)}\right) \ll \sum_{X_{\ell-1} < x_i \leqslant X_\ell}\sum_{j^* < j \leqslant J}\frac{1}{\log(y_j)}\int_{x_i/y_j}^{x_i/y_{j-1}}\E\left(|Y_{\frac{x_i}{z(1+1/\mathcal{X})},\frac{x_i}{z}}(z)|^4\right)\frac{dz}{z}.$$
We obtain the same bound on $\Prob\left(\mathcal{Q}_\ell^{(3)}\right)$ only using Chebychev's inequality and Cauchy--Schwarz (as there is no supremum to worry about). Now we only have to worry about bounding $\E(|Y_{\frac{x_i}{z(1+1/\mathcal{X})},\frac{x_i}{z}}(z)|^4)$. To do this, we appeal to Proposition \ref{prop_hyper} again:
\begin{align*}
    \E\left(\left|\sum_{\substack{n \leqslant z\\\frac{x_i}{z(1+1/\mathcal{X})} \leqslant P(n) < \frac{x_i}{z}}}\frac{f(n)}{\sqrt{n}}\right|^4\right) \leqslant& \left(\sum_{\substack{n \leqslant z \\ \frac{x_i}{z(1+1/\mathcal{X})} \leqslant P(n) < \frac{x_i}{z}}}\frac{\tau_3(n)}{n}\right)^2 \\
    \leqslant& \left(\sum_{\frac{x_i}{z(1+1/\mathcal{X})} < p \leqslant \frac{x_i}{z}}\frac{3}{p} \sum_{n \leqslant \frac{z}{p}} \frac{\tau_3(n)}{n}\right)^2 \\
    \ll& (\log(x_i))^6\left(\sum_{\frac{x_i}{z(1+1/\mathcal{X})} < p \leqslant \frac{x_i}{z}}\frac{1}{p}\right)^2.
\end{align*}
This is achieved by using the submultiplicativity of $\tau_3(n)$ and the simple bound
$$\sum_{n \leqslant x}\frac{\tau_3(n)}{n} \leqslant \sum_{\substack{n \geqslant 1 \\P(n) \leqslant x}}\frac{\tau_3(n)}{n} \ll \prod_{4 < p \leqslant x}\left(1-\frac{3}{p}\right)^{-1} \ll \log^3(x).$$ Next we observe that
$$\sum_{\frac{x_i}{z(1+1/\mathcal{X})} < p \leqslant \frac{x_i}{z}}\frac{1}{p} \ll \frac{1}{\mathcal{X}\log(y_j)}$$
as $x_i/y_j \leqslant z \leqslant x_i/y_{j-1}$. From this, we deduce that for $k=2,3$
$$\Prob\left(\mathcal{Q}_\ell^{(k)}\right) \ll \frac{\ell^K}{(T(\ell))^2}\sum_{X_{\ell-1} < x_i \leqslant X_\ell}\sum_{j^* < j \leqslant J}\frac{(\log(x_i))^6}{\mathcal{X}^2(\log(y_0))^2}.$$
By our choice of $\mathcal{X}$ (and $r$) this is more than sufficient to show that both $\Prob\left(\mathcal{Q}_\ell^{(2)}\right)$ and $\Prob\left(\mathcal{Q}_\ell^{(3)}\right)$ are summable in $\ell$.
\end{proof}
\section{Bounding $\mathcal{Q}_\ell^{(1)}$} \label{Q_1_part1}
This leaves us with handling the $\Prob\left(\mathcal{Q}_\ell^{(1)}\right)$ term, which will take by far the most effort in handling. We will explain why we seek to prove this type of bound and what can be done to improve this at the end of Sections \ref{Q_good} and \ref{setup_heur} since there are two places where improvements can be made.
\subsection{Preparing to apply Proposition \ref{HA1}}
To begin, we will change $\mathcal{Q}_\ell^{(1)}$ slightly in order to eventually apply Parseval's identity for Dirichlet series so we can eventually apply Harper's low moments results for random Euler products \cite{aharper201}, as well as to appropriately normalise this quantity so we can show it is a supermartingale. These modifications appear in both Caich \cite{caich23} and Mastrostefano's \cite{mastrostefano2022almost} work. First, we introduce a small shift in the $z$ power of $z^{1/2\log(X_\ell)}$. On the range $[x_i/y_j,x_i/y_{j-1}]$, this factor acts as a multiplicative constant. This allows us to conclude
$$\mathcal{Q}_\ell^{(1)}(x_i,y_j;f) \ll \frac{1}{\log(y_j)}\int_{x_i/y_j}^{x_i/y_{j-1}}\left|\sum_{\substack{n \leqslant z\\P(n) < \frac{x_i}{z}}}\frac{f(n)}{\sqrt{n}}\right|^2\frac{dz}{z^{1+\frac{2}{\log(X_\ell)}}}.$$
Then after extending the range of the integral, we see that $\mathcal{Q}_\ell^{(1)}(x_i,y_j;f) \ll U_j$ where
\begin{equation}
    U_{j} := \frac{1}{\log(y_j)}\int_{0}^\infty \max_{y_{j-1} < p \leqslant y_j}\left|\sum_{\substack{n \leqslant z\\P(n) <p}}\frac{f(n)}{\sqrt{n}}\right|^2 \frac{dz}{z^{1+\frac{2}{\log(X_\ell)}}}. \label{u_def}
\end{equation}
What is very important here is that $U_j$ is unaffected by the supremum over the $x_i$. Using this, we define the new event
\begin{equation}
    \mathcal{Q}_\ell^{(*)} := \left\{ \sup_{j^* \leqslant j \leqslant J} U_{j} > \frac{\ell^{K/2}T(\ell)}{C\ell\log(\ell)}\right\}, \label{Q*def}
\end{equation}
for some $C>0$ constant (to absorb the various implicit constants that we used when defining $U_j$). Due to the previous reasoning, we see that $\Prob(\mathcal{Q}_\ell^{(1)}) \leqslant \Prob(\mathcal{Q}_\ell^{(*)})$. Next we apply Parseval's identity for Dirichlet series to $U_{j}$, and show that this is a supermartingale.\par
\subsection{Handling the good part of $\mathcal{Q}_\ell^{(1)}$} \label{Q_good}
\begin{lemma}
    For $\ell$ sufficiently large, we have that the sequence
    \begin{equation}
        \mathcal{I}_{j} := \frac{1}{\log(y_j)}\int_{-\infty}^\infty \left|\frac{F_{y_{j}}(1/2+1/\log(X_\ell)+it)}{\frac{1}{\log(X_\ell)}+it}\right|^2dt
    \end{equation}
    is a supermartingale with respect to the filtration $(\mathcal{F}_{y_j})_{0 \leqslant j \leqslant J}$. \label{I_super}
\end{lemma}
\begin{proof}
    Following Caich, we have that
    \begin{align*}
        \E(\mathcal{I}_j \ | \mathcal{F}_{y_{j-1}}) =& \frac{1}{\log(y_{j})}\int_{-\infty}^{\infty}\E\left(\left|\frac{F_{y_{j}}(\frac{1}{2}+\frac{1}{\log(X_\ell)}+it)}{\frac{1}{\log(X_\ell)}+it}\right|^2 dt\ \bigg| \mathcal{F}_{y_{j-1}}\right) dt \\
        =& \frac{1}{\log(y_j)} \int_{-\infty}^{\infty}\E\left(\prod_{y_{j-1} < p \leqslant y_j}\left|1+\frac{f(p)}{p^{\frac{1}{2}+\frac{1}{\log(X_\ell)}+it}}\right|^2\right)\left|\frac{F_{y_{j-1}}(\frac{1}{2}+\frac{1}{\log(X_\ell)}+it)}{\frac{1}{\log(X_\ell)}+it}\right|^2  dt.
    \end{align*}
    Since $\sigma_\ell := \frac{1}{\log(X_\ell)} > 0$, we can deduce using the second part of Proposition \ref{euler_prod_result}
    \begin{align*}
    \E\left(\prod_{y_{j-1} < p \leqslant y_j}\left|1+\frac{f(p)}{p^{1/2+\sigma_\ell+it}}\right|^2\right) = \exp&\left(\sum_{y_{j-1} < p \leqslant y_{j}}\frac{1}{p^{1+2\sigma_\ell}} + O\left(\frac{1}{\sqrt{y_{j-1}}\log(y_{j-1})}\right)\right) \\ 
    = \exp&\Biggr(\sum_{y_{j-1} < p \leqslant y_{j}}\frac{1}{p} + \sum_{y_{j-1} < p \leqslant y_j}\frac{e^{-2\sigma_\ell\log(p)}-1}{p} \\
    \quad &+ O\left(\frac{1}{\sqrt{y_{j-1}}\log(y_{j-1})}\right)\Biggr) \\
    \leqslant \exp&\Biggr(\sum_{y_{j-1} < p \leqslant y_{j}}\frac{1}{p} - \sum_{y_{j-1} < p \leqslant y_j}\frac{2\log(p)}{p\log(X_\ell)} \\ & \quad + \sum_{y_{j-1} < p \leqslant y_j}\frac{2(\log(p))^2}{p(\log(X_\ell))^2} +O\left(\frac{1}{\sqrt{y_{j-1}}\log(y_{j-1})}\right)\Biggr) \\
    \ll \exp&\left(\frac{1}{\ell} - \frac{2e^{\frac{j-1}{\ell}-K\ell^{K-1}}}{\ell}+o\left(\frac{e^{\frac{j-1}{\ell}-K\ell^{K-1}}}{\ell}\right)\right).
    \end{align*}
    For sufficiently large $\ell$, the remainder terms are summable in $\ell$ and are negligible. This means that
    $$\E(\mathcal{I}_j \ | \mathcal{F}_{{y}_{j-1}}) \leqslant a(j)\mathcal{I}_{j-1}$$
    where
    $$a(j) = \exp\left(-\frac{Ce^{\frac{j-1}{\ell}-K\ell^{K-1}}}{\ell}\right) \leqslant 1,$$
    for some $C>0$ constant, which is sufficient to prove the supermartingale condition.
\end{proof}
Interestingly, we do not need the extra prefactor that Caich requires to make this a super martingale due to our shift to the right of the critical line. Before the next lemma, we define the events $\mathcal{S}_{j,\ell} := \{\mathcal{I}_j \leqslant \frac{\sqrt{T(\ell)}}{\ell^{K/2}\sqrt{\ell\log(\ell)}}\}$ and $\mathcal{S}_\ell := \bigcap_{j^* < j \leqslant J}\mathcal{S}_{j,\ell}$.
\begin{lemma}
    The sum $\sum_{l \geqslant 1}\Prob\left(\mathcal{Q}_\ell^{(*)}\right)$ converges, given that the sum of $\Prob(\overline{\mathcal{S}_\ell})$ converges. \label{Q*conv}
\end{lemma}
\begin{proof}
By the triangle inequality, we have that
\begin{align*}
    \Prob\left(\mathcal{Q}_\ell^{(*)}\right) \leqslant& \Prob\left( \sup_{j^* \leqslant j \leqslant J}\left\{U_{j} > \frac{\ell^{K/2}T(\ell)}{\ell\log(\ell)}\right\} \bigcap \{\mathcal{S_\ell}\}\right) + \Prob(\{\overline{\mathcal{S}_\ell}\}) \\
    \leqslant& \sum_{j^* \leqslant j \leqslant J}\Prob\left(\left\{U_j \geqslant \frac{\ell^{K/2}T(\ell)}{\ell\log(\ell)} \right\} \bigcap \{\mathcal{S}_{j-1,\ell}\}\right) + \Prob(\{\overline{\mathcal{S}_\ell}\}). \numberthis \label{Q*ineq}
\end{align*}
We will handle the first event now. Notice that after applying Markov's inequality, we have
$$\Prob\left(\left\{U_j \geqslant \frac{\ell^{K/2}T(\ell)}{\ell\log(\ell)} \right\} \bigcap \{\mathcal{S}_{j-1,\ell}\}\right) \leqslant \frac{\ell\log(\ell)}{\ell^{K/2}T(\ell)}\E(U_j | \mathcal{S}_{j-1,\ell}).$$
We now can use Lemma \ref{Y_sub} combined with Doob's $L^2$ inequality (Proposition \ref{Lr-Doob} for $r=2$) to deduce
\begin{align*}
    \E\left(\max_{y_{j-1} < p \leqslant y_j}\left|\sum_{\substack{n \leqslant z\\P(n)<p}}\frac{f(n)}{\sqrt{n}}\right|^2 \ \bigg| \mathcal{S}_{j-1,\ell}\right) \leqslant& \ 4 \max_{y_{j-1} < p \leqslant y_j}\E\left(\left|\sum_{\substack{n \leqslant z\\P(n) < p}}\frac{f(n)}{\sqrt{n}}\right|^2 \ \bigg| \mathcal{S}_{j-1,\ell}\right) \numberthis \label{bad}\\
    =& \ 4 \E\left(\left|\sum_{\substack{n \leqslant z\\P(n) < y_j}}\frac{f(n)}{\sqrt{n}}\right|^2 \ \bigg| \mathcal{S}_{j-1,\ell}\right).
\end{align*}
We resume computing the expectation of $U_j$ conditional on $\mathcal{S}_{j-1,\ell}$:
\begin{align*}
    \E(U_j | \mathcal{S}_{j-1,\ell}) =& \frac{1}{\log(y_j)}\int_{0}^\infty \E\left(\max_{y_{j-1} < p \leqslant y_j}\left|\sum_{\substack{n \leqslant z\\P(n) < p}}\frac{f(n)}{\sqrt{n}}\right|^2 \ \bigg| \mathcal{S}_{j-1,\ell}\right) \frac{dz}{z^{1+2\sigma_\ell}} \\
    \ll& \E\left(\frac{1}{\log(y_j)}\int_0^\infty\left|\sum_{\substack{n \leqslant z\\P(n) < y_j}}\frac{f(n)}{\sqrt{n}}\right|^2\frac{dz}{z^{1+2\sigma_\ell}} \ \bigg| \mathcal{S}_{j-1,\ell}\right) \\
    =& \E(\mathcal{I}_j | \mathcal{S}_{j-1,\ell}).
\end{align*}
One can clearly see we have applied Proposition \ref{HA1} to obtain the final equality. We can now apply Lemma \ref{I_super} and we obtain the following chain of inequalities:
$$\E(U_j | \mathcal{S}_{j-1,\ell}) \ll \E(\mathcal{I}_j | \mathcal{S}_{j-1,\ell}) \leqslant \E(\mathcal{I}_{j-1} | \mathcal{S}_{j-1,\ell}) \leqslant \frac{\sqrt{T(\ell)}}{\ell^{K/2}\sqrt{\ell\log(\ell)}}$$
Again, we can follow Caich's proof and we find
\begin{align*}
    \sum_{j^* < j \leqslant J}\Prob\left(\left\{U_j > \frac{\ell^{K/2}T(\ell)}{\ell\log(\ell)}\right\}\bigcap\{\mathcal{S}_{j-1,\ell}\}\right) \leqslant& \sum_{j^* < j \leqslant J}\frac{\ell\log(\ell)}{\ell^{K/2}T(\ell)}\E(U_j|\mathcal{S}_{j-1,\ell}) \\
    \leqslant \sum_{j^* < j \leqslant J} \frac{\sqrt{\ell\log(\ell)}}{(T(\ell))^{1/2}\ell^K} \ll \frac{\sqrt{\ell\log(\ell)}}{(T(\ell))^{1/2}},
\end{align*}
which given our choice of $T(\ell)$ is summable in $\ell$. 
\end{proof}
At this point, we observe that this good event is one of the barriers to obtaining the sharp almost sure upper bound, as every other complement event is bounded above by $\frac{\sqrt{T(\ell)}}{\ell^{K/2}\sqrt{\ell\log(\ell)}}$ almost surely (which we prove in the previous Lemmas and the next section). This is caused by taking the upper bounding $\mathcal{Q}_\ell^{(1)}$ by $U_j$ which lost the dependence on $i$, which means that we must consider the supremum over \textit{all} the $y_j$ and not just the ones which are "close" to the individual $x_i$. The barrier to doing this is handling the part of the sum where the largest and second largest distinct prime factors of integer $n$ are of similar size. This explains why we are able to obtain a much sharper result when restricting to a unique large prime factor in Theorem \ref{sharp_restricted_ub}.
\subsection{Preparing to bound the bad part of $\mathcal{Q}_\ell^{(*)}$} \label{setup_heur}
Finally, we come to bounding $\Prob(\overline{\mathcal{S}_\ell})$. This will use the ideas of Hardy \cite{shardy23} and Gerspach \cite{gerspach20}, as well as Harper's Gaussian random walk and multiplicative chaos results \cite{harper23} to achieve our goal. Before doing this, we will define another event $\mathcal{R}_\ell:=\{\mathcal{I}_0 \leqslant \frac{(T(\ell))^{1/4}}{\ell^{K/2}(\ell\log(\ell))^{1/4}}\}$. We proceed similarly to before, and observe
\begin{align*}
    \Prob(\overline{\mathcal{S}_\ell}) \leqslant& \ \Prob\left(\left\{\max_{1 \leqslant j \leqslant J}\mathcal{I}_j > \frac{(T(\ell))^{1/2}}{\ell^{K/2}\sqrt{\ell\log(\ell)}} \right\} \ \bigg| \mathcal{R}_\ell\right) + \Prob(\overline{\mathcal{R}_\ell}) \numberthis \label{S_ineq} 
\end{align*}
We then use the fact that $(\mathcal{I}_j)_{j \leqslant J}$ is a supermartingale, and apply Proposition \ref{Doob_super} to show 
$$\Prob\left(\left\{\max_{0 \leqslant j \leqslant J}\mathcal{I}_j > \frac{(T(\ell))^{1/2}}{\ell^{K/2}\sqrt{\ell\log(\ell)}} \right\} \ \bigg| \mathcal{R}_\ell\right) \leqslant \frac{\sqrt{\ell\log(l)}\ell^{K/2}}{\sqrt{T(\ell)}}\E(\mathcal{I}_0 | \mathcal{R}_\ell) \leqslant \frac{(\ell\log(\ell))^{1/4}}{(T(\ell))^{1/4}},$$
which is summable in $\ell$ by the choice of $T(\ell)$. \par
To handle the final probability $\Prob(\overline{\mathcal{R}_\ell})$, we apply Markov's inequality to the exponent $q=2/3$. This obtains 
\begin{equation}
    \Prob(\overline{\mathcal{R}_\ell}) \leqslant \frac{\E\left[(\mathcal{I}_0^{2/3})\right](\ell\log(\ell))^{1/6}\ell^{K/3}}{(T(\ell))^{1/6}}. \label{bad_prob_bound}
\end{equation}
The rest of this section and the next section is dedicated to bounding $\mathcal{I}_0$. We will use the methods employed by Hardy \cite{shardy23}, which were inspired by the splitting used by Gerspach in his paper on pseudomoments of the Riemann zeta function \cite{gerspach20}. At this point, we split the domain of integration in $\mathcal{I}_0$ over various intervals to maximise our savings:
\begin{align}
        \int_{-\infty}^{\infty} S(t) dt \ll& \int_{-1/\log(y_0)}^{1/\log(y_0)} S(t) dt + \sum_{\substack{(\log(y_0))^{-1}< |T| \leqslant (\log_2(y_0))^{-K} \\ T \text{ dyadic}}} \int_T^{2T} S(t) dt \\
        +& \sum_{\substack{(\log_2(y_0))^{-K}< |T| \leqslant 1/1024 \\ T \text{ dyadic}}} \int_T^{2T} S(t) dt+ \int_{|t| > 1/512} S(t) dt \label{dyadic_decomp}
\end{align}
with
\begin{equation*}
    S(t) := \left|\frac{F_{y_0}(1/2+\sigma_\ell + it)}{\sigma_\ell+it}\right|^2
\end{equation*}
defined for compactness of notation. This looks very similar to the unweighted case, with the main difference coming from the different weighting from the denominator. In Harper's work \cite{aharper201}, the contribution from $t$ very close to $0$ could be handled using an application of H\"{o}lder's inequality in the Rademacher case, which we cannot do here. The other difference is that we want to find how large this object is almost surely, and not in expectation. This grants us various savings when handling the Rademacher Euler product (which can also be seen in a different sense when comparing the results of \cite{gerspach20} and \cite{shardy23} in the Steinhaus case). We will do this by introducing various sub-events to handle this event. The choice of stopping the dyadic cutting procedure at $t=1/(\log_2(y_0))^{\epsilon}$ is rather contrived, but we do this to ensure that all of our results converge in $\ell$ (this is important considering our estimate on the almost sure size of the random Euler product on various dyadic intervals).\par
The notation connected to the second and third integrals on the right hand side of equation (\ref{dyadic_decomp}) is rather confusing, so we explain it now. The $T$ considered are of the form $T = 2^n/\log(y_0)$, so they are in the range $|T| \in [1/\log(y_0),1/(\log_2(y_0))^{\epsilon}]$ (we are technically summing over $n$ here, but for tidiness we reduce it to $T$ only). The first limit point of $1/\log(y_0)$ is achieved using this dyadic procedure, but the other two limit points will be over and underestimated (but only affects the multiplicative constant). We discretise the range in order to get a much stronger control on the random Euler products. We now manipulate the integrals to the point where we can condition on the size of the integrals and the Euler products:
\begin{align*}
    \int_{-1/\log(y_0)}^{1/\log(y_0)}S(t) dt \leqslant (\log(X_\ell))^2&\int_{-1/\log(y_0)}^{1/\log(y_0)}\left|\frac{F_{y_0}(1/2+\sigma_\ell+it)}{F_{y_0}(1/2+\sigma_\ell)}\right|^2dt |F_{y_0}(1/2+\sigma_\ell)|^2;
\end{align*} and the appropriate form for the second integral being
\begin{align*}
    \sum_{\substack{1/\log(y_0)< |T| \leqslant (\log_2(y_0))^{-K}  \\ T \text{ dyadic}}}&\text{sgn(T)}\int_T^{2T} S(t) dt \\
    \leqslant  \sum_{\substack{1/\log(y_0)< |T| \leqslant (\log_2(y_0))^{-K} \\ T \text{ dyadic}}}&\frac{\text{sgn(T)}}{T^2} \int_T^{2T}\left|\frac{F_{y_0}(1/2+\sigma_\ell+it)}{F_{e^{1/|T|}}(1/2+\sigma_\ell)}\right|^2 dt |F_{e^{1/|T|}}(1/2+\sigma_\ell)|^2.
\end{align*}
We introduced the signum function for tidiness so we ensure that the integral limits will be written correctly when $T$ is negative. In the first two integrals, we get the saving $\ell^{K/2}$ from conditioning on the value of $|F_{e^{1/|T|}}(1/2+\sigma_\ell)|^2$. The third contribution is by far the most involved to deal with because we must be much more delicate when dealing with the denominator, so we use a combination of Gerspach and Harper's low moments methods \cite{aharper201} to get our desired bounds. For the final contribution, we directly use Proposition \ref{mult_chaos} to address this. We also split the integral up in the final contribution, and find that it can be upper bounded by
\begin{align*}
    \int_{|t| > 1/512} S(t) dt \ll \sum_{N \in \mathbb{Z}}\frac{1}{(N-1/2)^2}\int_{N-1/2}^{N+1/2}|F_{y_0}(1/2+\sigma_\ell+it)|^2dt.
\end{align*}
This prepares the integral for the use of Proposition \ref{mult_chaos}. Splitting this up further, we can see that if we split the contribution of $|N| \leqslant \log(y_0)$ and $|N| > \log(y_0)$, then it is sufficient to only use a mean square estimate for the Euler product since we would now have
\begin{align*}
    &\E\left(\sum_{|N| > \log(y_0)}\frac{1}{(N-1/2)^2}\int_{N-1/2}^{N+1/2}|F_{y_0}(1/2+\sigma_\ell+it)|^2 dt\right) \\ \ll& \frac{1}{\log(y_0)}\sup_{|t| > \log(y_0)-1/2}\left(\E|F_{y_0}(1/2+\sigma_\ell+it)|^2\right) \ll 1
\end{align*}
by upper bounding the convergent sum and Fubini--Tonelli. Then by Markov's inequality, we see that
$$\Prob\left(\sum_{|N| > \log(y_0)}\frac{1}{(N-1/2)^2}\int_{N-1/2}^{N+1/2}|F_{y_0}(1/2+\sigma_\ell+it)|^2 dt> \ell^2\right) \ll \frac{1}{\ell^2},$$
which is summable in $\ell$. We now define the events
\begin{align*}
    \mathcal{H}_{\ell}(T) =& \left\{\left|F_{e^{1/|T|}}(1/2+\sigma_\ell)\right| > A|T|^{1/10}\right\}; \\
    \mathcal{P}_{\ell}^{(1)} =& \left\{\int_{-1/\log(y_0)}^{1/\log(y_0)}\left|\frac{F_{y_0}(1/2+\sigma_\ell+it)}{F_{y_0}(1/2+\sigma_\ell)}\right|^2 dt > \frac{\ell^2}{\log(y_0)}\right\}; \\
    \mathcal{P}_{\ell}^{(2)} =& \left\{\frac{1}{\log(y_0)}\sum_{\substack{(\log(y_0))^{-1} \leqslant |T| \leqslant (\log_2(y_0))^{-K} \\ T \text{ dyadic}}}\frac{\text{sgn}(T)}{T^2} \int_{T}^{2T} \left|\frac{F_{y_0}(1/2+\sigma_\ell+it)}{F_{e^{1/|T|}}(1/2+\sigma_\ell)}\right|^2 dt > \ell^{K+2}\right\}; \\
    \mathcal{P}_l^{(3)} =& \left\{\frac{1}{\log(y_0)}\sum_{\substack{(\log_2(y_0))^{-K} \leqslant |T| \leqslant 1/1024 \\ T \text{ dyadic}}}\frac{\text{sgn}(T)}{T^2} \int_{T}^{2T}\left|F_{y_0}(1/2+\sigma_\ell+it)\right|^2 dt > \frac{\ell^{4}}{\ell^{K/2}}\right\}; \\
    \mathcal{P}_{\ell}^{(4)} =& \left\{\frac{1}{\log(y_0)}\sum_{|N| \leqslant \log(y_0)}\frac{1}{(N-1/2)^2}\int_{N-1/2}^{N+1/2} |F_{y_0}(1/2+\sigma_\ell+it)|^2dt > \frac{\ell^{2}}{\ell^{K/2}}\right\}; \\
    \mathcal{P}_\ell^{(5)} =& \left\{\sum_{|N| > \log(y_0)}\frac{1}{(N-1/2)^2}\int_{N-1/2}^{N+1/2}|F_{y_0}(1/2+\sigma_\ell+it)|^2 dt> \ell^2\right\};
\end{align*}
for $A>0$ an arbitrarily large constant. The range of $T$ we consider in the event $\mathcal{H}_\ell(T)$ is
\begin{equation}
    |T| \in [(\log(y_0))^{-1},(\log_2(y_0))^{-K}]. \label{T_range}
\end{equation}
We will show that all these events occur with probability summable in $\ell$ (and conclude using the first Borel--Cantelli lemma). These calculations will also be useful when looking at the proof of Theorem \ref{lbtheorem} later. After applying the conditioning on the events $\mathcal{H}_\ell(T)$ and each $\mathcal{P}_\ell^{(k)}$ not occurring, we have 
\begin{align*}
    \frac{1}{\log(y_0)}\int_{-\infty}^{\infty}S(t) dt \ll& \frac{\log(X_\ell))^2\ell^2}{(\log(y_0))^{11/5}} + \ell^{K+2}(\log\log(y_0))^{-K/5}\numberthis \label{big_decomp}\\ +&  \frac{\ell^4}{\ell^{K/2}} + \frac{\ell^{2}}{\ell^{K/2}} + \ell^{2}.
\end{align*}
We see that $\mathcal{I}_0$ satisfies the upper bound
$$\ll \left(\frac{(\log(X_\ell))^2\ell^2}{(\log(y_0))^{11/5}}+\ell^2\ell^{-cK^2} +\ell^2\ell^{-K/2} +\ell^{4}\ell^{-K/2}+\ell^{2}\right),$$
for $c>0$ a small constant. This can be further upper bounded by
$$B\ell^{4}\ell^{-K/2}$$
with $B>0$ an absolute large constant. Then this can be substituted into equation (\ref{bad_prob_bound}), where we now see that
$$\Prob(\overline{\mathcal{R}_\ell}) \leqslant \frac{B\ell^{8/3}\ell^{-K/3}(\ell\log(\ell))^{1/6}\ell^{K/3}}{(T(\ell))^{1/6}} \ll \frac{(\log(\ell))^{1/6}}{\ell^{7/6}},$$
which is summable in $\ell$, which then by the first Borel--Cantelli lemma, shows that the event $\mathcal{R}_{\ell}$ almost surely holds for sufficiently small $\epsilon$. All that is now left is to prove that each of the previously defined events are also summable in $\ell$, so only occur finitely often.
\section{Bounding the bad part of $\mathcal{Q}_\ell^{(*)}$} \label{ce}
All that is now left is to show that for any $\epsilon>0$, all the terms on the right hand side of equation (\ref{big_decomp}) are summable in $\ell$. 
\subsection{Bounding $\Prob(\mathcal{H}_\ell(T))$ and $\Prob(\mathcal{H}_\ell')$} \label{euler_LOIL}
Given that the shift $\sigma_\ell$ is so small it is negligible, it suffices to bound $\Prob(\mathcal{H}_\ell(T))$ as $\Prob(\mathcal{H}_\ell(T))$ can be handled in a similar fashion (the shift is $\frac{1}{\log(X_\ell)}$, which can be dealt with using partial summation, and differs by a multiplicative error). We recall that
$$\mathcal{H}_\ell ' := \left\{ \left|F_{y_{j^*}}(1/2)\right| > A(\log(y_{j^*}))^{-1/10}\right\}.$$
In Hardy's treatment of the Steinhaus case, this event had to be treated very delicately, since most of the attention was dedicated to handling the sum of independent random variables
\begin{equation*}
    \sum_{p \leqslant x}\frac{f_{st}(p)}{\sqrt{p}}
\end{equation*}
where $f_{st}$ are Steinhaus random variables. This required a second range of sparse points to separate the $X_\ell$ to attain the power saving desired in that bound. As for our case, we do not need to be as delicate since we have that
\begin{equation}
    F(s) = \prod_p\left(1+\frac{f(p)}{p^s}\right) = \exp\left(\sum_p \frac{f(p)}{p^s} - \sum_p\frac{1}{2p^{2s}} + O(1)\right). \label{euler_exp}
\end{equation}
We will see that the second term in the second equality is the dominant contribution when evaluating the size of this Euler product. The bound we seek is certainly not optimal in any regard, with a better bound being likely achievable using the methods of \cite{shardy23}. However, the bound we seek is easier to prove and allows us to consider the entire range stated in equation (\ref{T_range}), so we do not pursue this here. \par
Note that $y_{j^*}$ is even larger than $y_0$, so a similar method applies with even stronger summability. Taking the logarithm on both sides and Taylor expanding the logarithm shows that equation (\ref{euler_exp}) is upper bounded by:
\begin{equation*}
    \Prob\left(\sum_{p \leqslant e^{1/|T|}} \frac{f(p)}{p^{1/2+\sigma_\ell}} - \sum_{p \leqslant e^{1/|T|}} \frac{1}{2p^{1+2\sigma_\ell}} > \frac{1}{10}\log\left(\frac{1}{|T|}\right) + A'\right),
\end{equation*}
($A>A'>0$ is again an arbitrarily large constant). Since $\sum_{p \leqslant e^{1/|T|}}\frac{1}{p^{1+2\sigma_\ell}} = \log(\frac{1}{|T|}) +O(1)$, this leaves us with the upper bound
\begin{equation}
    \Prob\left(\sum_{p \leqslant e^{1/|T|}}\frac{f(p)}{p^{1/2+\sigma_\ell}}> \frac{2}{5}\log\left(\frac{1}{|T|}\right) \right). \label{Gcomp}
\end{equation}
One sees that the $A$ introduced in the definition of $\mathcal{H}_\ell(T)$ is introduced to absorb the various $O(1)$ terms. Appealing to Proposition \ref{chernoff} with choice $x = \frac{2}{5}\left(\log\left(\frac{1}{|T|}\right)\right)^{1/2}$ (which is permissible), we obtain the above satisfies
$$\ll \exp\left(-0.128\log\left(\frac{1}{|T|}\right)\right).$$
Note that we only are interested in this quantity when $\exp(1/|T|) \gg \exp((\log\log(y_0))^{\epsilon})$, so this final bound has size at most
$$\exp\left(-0.128K\epsilon\log(\ell)\right) \ll \ell^{-3}$$
which is summable in $\ell$.
\subsection{Bounding the integral events}
Before proceeding, We remark that we already showed that $\Prob(\mathcal{P}_\ell^{(5)})$ is summable in $\ell$ in Section \ref{setup_heur}, which leaves us with four more events to consider. For the first two events, we appeal to Proposition \ref{euler_prod_result}. This is why we chose the lengths $e^{1/|T|}$ for the Euler products in the definition of $\mathcal{P}_\ell^{(2)}$ and $\mathcal{P}_\ell^{(3)}$, so the range including the small primes can get appropriately cancelled out. \par
We start with $\Prob\left(\mathcal{P}_{\ell}^{(1)}\right)$. Applying Markov's inequality and the union bound, we find that using Proposition \ref{euler_prod_result} and equation (\ref{short_euler_cancel})
\begin{align*}
    \Prob\left(\mathcal{P}_{\ell}^{(1)}\right) \leqslant&   \frac{\log(y_0)}{\ell^2}\int_{-1/\log(y_0)}^{1/\log(y_0)} \E\left(\left|\frac{F_{y_0}(1/2+\sigma_\ell +it)}{F_{y_0}(1/2+\sigma_\ell)}\right|^2\right) dt \\
    \ll&   \frac{\log(y_0)}{\ell^2} \int_{-1/\log(y_0)}^{1/\log(y_0)} dt \\
    \ll&  \frac{1}{\ell^2}.
\end{align*}
For $\Prob\left(\mathcal{P}_{\ell}^{(2)}\right)$, we proceed similarly and find that
\begin{align*}
    \Prob\left(\mathcal{P}_{\ell}^{(2)}\right) \leqslant&  \frac{1}{\ell^{K+2}\log(y_0)} \sum_{\substack{(\log(y_0))^{-1} \leqslant |T| \leqslant (\log_2(y_0))^{-K} \\ T \text{ dyadic}}}\frac{1}{T^2} \int_{T}^{2T} \E\left(\left|\frac{F_{y_0}(\frac{1}{2}+\sigma_\ell+it)}{F_{e^{1/|T|}}(1/2+\sigma_\ell)}\right|^2\right) dt \\
    \leqslant& \frac{1}{\ell^{K+2}\log(y_0)}\sum_{\substack{(\log(y_0))^{-1} \leqslant |T| \leqslant (\log_2(y_0))^{-K} \\ T \text{ dyadic}}} \frac{1}{T^2} \int_{T}^{2T} \E\left(\left|\frac{F_{e^{1/|T|}}(\frac{1}{2}+\sigma_\ell+it)}{F_{e^{1/|T|}}(1/2+\sigma_\ell)}\right|^2\right) \\&\times\E\left(\prod_{e^{1/|T|} < p \leqslant y_0}\left|1+\frac{f(p)}{p^{1/2+\sigma_\ell+it}}\right|^2\right) dt \\
    \ll& \frac{1}{\ell^{K+2}\log(y_0)}\sum_{\substack{(\log(y_0))^{-1} \leqslant |T| \leqslant (\log_2(y_0))^{-K} \\ T \text{ dyadic}}} \frac{T\log(y_0)\cdot T}{T^2} 
    \ll \frac{1}{\ell^2}.
\end{align*}
Here we have used the fact the individual factors of the product are independent, and a mean square estimate. Next, we proceed by analysing $\Prob\left(\mathcal{P}_\ell^{(4)}\right)$. We need a slightly different result to prove this, namely Proposition \ref{mult_chaos}. This allows us to apply Markov's inequality with exponent $q=\frac{2}{3}$ for example with $X =y_0$. Then we have that
\begin{align*}
    \Prob\left(\mathcal{P}_{\ell}^{(4)}\right) \leqslant \left(\frac{\sqrt{\log_2(y_0)}}{\ell^{2}\log(y_0)}\right)^{2/3}\E\left[\left(\sum_{|N| \leqslant \log(y_0)}\frac{1}{(N-1/2)^2}\int_{N-1/2}^{N+1/2} |F_{y_0}(1/2+\sigma_\ell+it)|^2dt\right)^{2/3}\right].
\end{align*}
We then apply Fubini--Tonelli to exchange the order of summation and expectation, the inequality $(a+b)^{2/3} \leqslant a^{2/3}+b^{2/3}$ when $a,b >0$ and finally apply Proposition \ref{mult_chaos} with $X=y_0$ to find
\begin{align*}
    \Prob\left(\mathcal{P}_\ell^{(4)}\right) \leqslant& \left(\frac{\sqrt{\log_2(y_0)}}{\ell^2\log(y_0)}\right)^{2/3}\sum_{|N| \leqslant \log(y_0)}\E\left[\left(\frac{1}{(N-1/2)^2}\int_{N-1/2}^{N+1/2}|F_{y_0}(1/2+\sigma_\ell+it)|^2dt\right)^{2/3}\right] \\ 
    \ll& \sum_{|N| \leqslant \log(y_0)}\frac{((\log_2(|N|+10)))^{2/3}}{\ell^{4/3}(N-1/2)^{4/3}}\\ \ll& \frac{\log(\ell)}{\ell^{4/3}}
\end{align*}
which is summable in $\ell$.
\subsection{Random Euler products and Gaussian walks; bounding $\Prob\left(\mathcal{P}_\ell^{(3)}\right)$} \label{tilted_prob}
Our final task is to bound
$$\sum_{\substack{(\log_2(y_0))^{-K} \leqslant |T| \leqslant 1/1024 \\ T \text{ dyadic}}}\frac{\text{sgn}(T)}{T^{2}}\int_{T}^{2T}\left|F_{y_0}(1/2+\sigma_\ell+it)\right|^2dt.$$
This is trickier than the other sections, and consequently we use much more powerful tools. This section is almost identical to Key Proposition 3 in Harper's low moments paper, with a slight change due to the definition of the tilted measure (we will consider each dyadic interval individually and then sum over all of them at the end). 
To do this, we need the Gaussian walk machinery developed in said paper as well as a splitting argument of Gerspach. We will also impose a barrier condition on the Euler product to ensure it behaves sufficiently well. To begin, we define the tilted probability measure
$$\widetilde{\Prob}_{t}(A) := \frac{\E\mathbf{1}_A\prod_{e^{1/|T|} < p \leqslant y_0^{1/e}}\left|1+\frac{f(p)}{p^{1/2+\sigma_\ell+it}}\right|^2}{\E\prod_{e^{1/|T|} < p \leqslant y_0^{1/e}}\left|1+\frac{f(p)}{p^{1/2+\sigma_\ell+it}}\right|^2}.$$
where $T$ is the largest point of the form $\frac{2^n}{\log(y_0)} < |t|$. This titled measure is introduced in order to take the integral average over $t$ and the expectation over the Rademacher $f(p)$ simultaneously. While it is a different choice of tilted measure to that which Harper chooses in his work, the range of primes can be adjusted to match those which appear in the event $A$. The other primes are then independent of the event $A$, so then can be pulled out of the conditioned expectation, and cancel with the corresponding prime in the expectation in the denominator. The event we choose to condition on only concerns primes larger than $e^{1/|T|}$ (and smaller than $y_0^{1/e}$), so we are allowed to choose our tilted probability measure like this while retaining the probability results of Harper. Then let
$$I_k(s) = \prod_{y_0^{e^{-(k+2)}} < p \leqslant y_0^{e^{-(k+1)}}}\left(1+\frac{f(p)}{p^s}\right)$$
be the $k$-th increment of the Rademacher Euler product. We now state Harper's Proposition 6 in \cite{aharper201} for convenience. We treat all the probability ideas used to prove this as a black box. This was proved by showing that the $I_k(s)$ can be well approximated by independent Gaussian random variables with mean $1$ and variance $1$, and then manipulating things further results similar to Proposition \ref{ballot}. In order for this bound to have a significant contribution to the size of the bound, we need the number of $k$ to be approximately of size $C\log\log(y_0)$ for some $C>0$. Over the range of $T$ we are considering, then we will have at least $\log_2(y_0) - K\log_3(y_0) \gg \log_2(y_0)$ such $k$ to consider which will give us the saving of $\ell^{K/2}$ we desire, which we see in the following proposition.
\begin{definition}
    (Harper, Proposition 6) \textit{There exists a large natural number $B$ such that the following is true.} \par
    \textit{Let $t \in \mathbb{R}$ and $D \geqslant \max\{\log(1/|t|), 2\log_2(1+|t|)\}+B+1$ be any natural number. Let $n \leqslant \lfloor\log_2(y_0)\rfloor - D$ be large and define the decreasing sequence of points $(k_j)_{j=1}^n$ by $k_j = \lfloor\log_2(y_0)\rfloor - D -j$. Suppose $|\sigma| < e^{-(D+n)}$ and $(t_j)_{j=1}^n$ is a sequence of real numbers satisfying $|t_j-t| \leqslant j^{-2/3}e^{-(D+j)}$.} \par
    \textit{Then uniformly for large $a$ and function $h$ satisfying $|h(n)| \leqslant 10\log(n)$, we have}
    $$\widetilde{\Prob}_t\left(-a-Bj \leqslant \sum_{m=1}^j\log|I_{k_m}(1/2+\sigma+it_m)| \leqslant a + j + h(j) \ \forall j \leqslant n\right) \asymp \min\{1,\frac{a}{\sqrt{n}}\}.$$ \label{gaussian_walk}
\end{definition}
A very important point is that when $k \in \mathbb{N}$, then $|I_k(s)|$ is well approximated by Gaussian random variable $G_k$. Each of these have mean $1$ and variance $1$, and are independent of each other. We will choose $D(t) = \lceil\log(1/|t|)\rceil + B + 1$ (which satisfies the hypotheses of Key Proposition $2$ for this range of $t$). A suitable set $t(j)$ is the sequence of points with $t(-1) =t$ and
$$t(j) := \max\left\{u \leqslant t(j-1): u = \frac{n}{e^{-(j+1)}\log(y_0)\log(\log(y_0)e^{-(j+1)})} \ \text{ for some } n \in \mathbb{Z} \right\}.$$
(The definition of $t(j)$ is corresponds to $t_{k_j}$ in the above proposition. This has been introduced to keep the definition of the event clearer). Then we define the event $\mathcal{G}_{\ell}^{\text{Rad}}(t)$ where we have that for all $1 \leqslant j \leqslant\lfloor \log_2(y_0)\rfloor - D - 1$  
\begin{equation}
    \left(\frac{\log(y_0)}{e^{j+1}}e^{g(j,y_0)}\right)^{-1} \leqslant \prod_{m=j}^{\lfloor\log_2(y_0)\rfloor - D -1} |I_m(1/2+\sigma_\ell+it(m))| \leqslant \left(\frac{\log(y_0)}{e^{j+1}}e^{g(j,y_0)}\right) \label{main_event}
\end{equation}
where $g(j,x) = C\min\{\sqrt{\log_2(x)},\frac{1}{1-q}\} + 2\log_2(\log(x)e^{-(j+1)})$ for some large constant $C>0$. The smallest prime which the event $\mathcal{G}_\ell^{\text{Rad}}(t)$ conditions on is at least 
$$\exp\left(\log(y_0)e^{-\lfloor\log_2(y_0)\rfloor}e^{\log(1/|t|)+B+2}\right) \geqslant \exp\left(\frac{e^{B+2}}{|t|}\right)$$
with $B$ being a large natural number. In particular, this is larger than $e^{1/|T|}$, which justifies our choice of tilted measure. Then we denote $\mathcal{G}^{\text{Rad}}_\ell$ the event that $\mathcal{G}_\ell^{\text{Rad}}(t)$ occurs for each $t \in [-1/2,1/2]$, which is the barrier event we will use in the proceeding calculations. This matches the definition of $\mathcal{G}^{\text{Rad}}(1)$ in Harper's work, except for the fact that this event now explicitly depends on $\ell$ and a different $\sigma_\ell$ value (which is valid for the application of Key Proposition 2 again). Then we see that
\begin{equation}
    \widetilde{\Prob}_t\left(\mathcal{G}_\ell^{\text{Rad}} \text{ fails}\right) \ll e^{-2C\min\{\sqrt{\log_2(y_0)},\frac{1}{1-q}\}} \label{fail_event}
\end{equation}
for the same $C>0$ stated in equation (\ref{main_event}). This  bound is uniform in $k$ in Key Proposition 4, so it holds for us too. Consequently, we need to bound
$$\E\left(\mathbf{1}_{\mathcal{G}_\ell^{\text{Rad}}}\int_{T}^{2T}\left|F_{y_0}(1/2+\sigma_k+it)\right|^2 dt\right)^{q}$$
for $q \in (0,1/2)$. In view of H\"{o}lder's inequality and that $\mathbf{1}_{\mathcal{G}_\ell^{\text{Rad}}} \leqslant \mathbf{1}_{\mathcal{G}_\ell^{\text{Rad}}(t)}$, it is sufficient to bound
$$\mathbb{E}\left(\int_T^{2T}\mathbf{1}_{\mathcal{G}_\ell^{\text{Rad}}(t)} \left|F_{y_0}(1/2+\sigma_k +it)\right|^2dt\right)^q.$$
We can perform the same splitting trick as seen in Gerspach to approximate the full Euler product. This yields
$$\ll \left(\mathbb{E}\left(|F(1/2+\sigma_k)|^{2q}\right)\right)^{2(1-q)}\mathbb{E}\left[\left(\frac{1}{\log(x)}\int_{T}^{2T}\mathbf{1}_{\mathcal{G}_\ell^{\text{Rad}}(t)}\left|\frac{F_{y_0}(1/2+\sigma_k+it)}{(F_{e^{1/|T|}}(1/2+\sigma_k))^{1-q}}\right|^2dt\right)^q\right]$$
We can take use Holder's inequality to consider the first moment of the integral to the $q$-th power, and then independence so we can evaluate the contribution of the integral of the long Euler product and the ratios of Euler products separately. This leaves us with bounding the above by
\begin{align*}&\left(\mathbb{E}\left(|F(1/2+\sigma_k)|^{2q}\right)\right)^{2(1-q)} \mathbb{E}\left(\left|F_{e^{1/|T|}}(1/2+\sigma_k+it)\right|^{2q}\left|\frac{F_{e^{1/|T|}}(1/2+\sigma_k+it)}{F_{e^{1/|T|}}(1/2+\sigma_k)}\right|^{2(1-q)}\right)^q\\
&\times \left(\frac{1}{\log(x)}\int_T^{2T}\mathbb{E}\left(\mathbf{1}_{\mathcal{G}_\ell^{\text{Rad}}(t)}\prod_{e^{1/|T|} < p \leq y_0}\left|1+\frac{f(p)}{p^{1/2+\sigma_k+it}}\right|^2\right)dt\right)^q.
\end{align*}
We are allowed to drop the indicator functions from the first two expectations since we are only conditioning on contributions associated with larger primes. Since we have already considered the contribution of the first two terms, we will only look at the third term in detail.
For now, we will look only at the integrand. We obtain
\begin{align*}
    &\E\left(\mathbf{1}_{\mathcal{G}_\ell^{\text{Rad}}(t)}\prod_{\exp(1/|T|) < p \leqslant y_0}\left|1+\frac{f(p)}{p^{1/2+\sigma_k+it}}\right|^2\right) \\ \leqslant \ & \E\left(\prod_{\exp(1/|T|) < p \leqslant y_0} \left|1+\frac{f(p)}{p^{1/2+\sigma_k+it}}\right|^2\right)\frac{\E\left(\mathbf{1}_{\mathcal{G}_\ell^{\text{Rad}}(t)}\prod_{e^{1/|T|} < p \leqslant y_0}\left|1+\frac{f(p)}{p^{1/2+\sigma_k +it}}\right|^2\right)}{\E\left(\prod_{e^{1/|T|} < p \leqslant y_0}\left|1+\frac{f(p)}{p^{1/2+\sigma_k+it}}\right|^2\right)} \\
    \ll \ & |T|\log(y_0) \widetilde{\Prob}_t(\mathcal{G}_\ell^{\text{Rad}}(t)) \ll \frac{D|T|\log(y_0)}{1+(1-q)\sqrt{\log_2(y_0))}},
\end{align*}
where we finally apply Key Proposition \ref{gaussian_walk} on the final step with $a = C\min\{\sqrt{\log_2(x)},\frac{1}{1-q}\} + D + 2\log(D+1)$, $t_m = t(\lfloor\log_2(x)\rfloor - D -m)$, $h(j) = 2\log(j)$, $n = \lfloor\log_2(x)\rfloor - D -1$ and using that we have $q \in [0,1/2)$ in our case. There is also a constant absorbed when handling the primes such $e^{1/|T|} < p \leqslant e^{(B+2)/|T|}$, which we know can be upper bounded by $(B+2)$, a large natural number by Proposition \ref{euler_prod_result} (which is contained in the Vinogradov notation). Recall we still have to integrate $D(t)$ as well. Clearly, we have
$$\int_T^{2T}\log(1/|t|)dt \ll T\log(1/|T|).$$
In all, we have
\begin{align*}
    \E\left(\frac{1}{\log(y_0)}\mathbf{1}_{\mathcal{G}_\ell^{\text{Rad}}}\int_{T}^{2T}\left|F_{y_0}(1/2+\sigma_k+it)\right|^2 dt\right)^{q} 
    \ll \left(\frac{|T|^{3-2q}\log(1/|T|)}{1+(1-q)\sqrt{\log_2(y_0)}}\right)^q. \numberthis \label{good_event_gaussian}
\end{align*}
After this, we show that this implies that
$$\E\left(\frac{1}{\log(y_0)}\int_{T}^{2T}\left|F_{y_0}(1/2+\sigma_\ell+it)\right|^2 dt\right)^{q} \ll \left(\frac{|T|^{3-2q}\log(1/|T|)}{1+(1-q)\sqrt{\log_2(y_0)}}\right)^q$$
for $q \in [0,1]$ (more precisely, showing that the contribution of the failure event is negligible). To do this, we again follow the approach of Harper with a small modification. We perform the first two steps using H\"{o}lder's inequality to estimate the random Euler product by its size at the central point. This yields
\begin{align*}
    &\E\left(\frac{1}{\log(y_0)}\int_{T}^{2T}\left|F_{y_0}(1/2+\sigma_k+it)\right|^2 dt\right)^{q} \\
    \ll& |T|^{q-2q^2}\mathbb{E}\left[\left(\frac{1}{\log(y_0)}\int_T^{2T}\prod_{e^{1/|T|} < p \leq y_0}\left|1+\frac{f(p)}{p^{1/2+\sigma_k+it}}\right|^2dt\right)^q\right].
\end{align*}
At this point, we have the desired bound on the short Euler product, so we look only to the By H\"{o}lder's inequality, it is sufficient to prove this for the range $q \in [2/3,1-\frac{1}{\sqrt{\log_2(y_0)}}]$ to obtain the result for the range $[0,1]$. Then we have
\begin{align*}
    \E&\left(\frac{1+(1-q)\sqrt{\log_2(y_0)}}{\log(y_0)}\int_T^{2T}\prod_{e^{1/|T|} < p \leq y_0}\left|1+\frac{f(p)}{p^{1/2+\sigma_k+it}}\right|^2dt\right)^q \\ \leqslant \ \E&\left(\frac{1+(1-q)\sqrt{\log_2(y_0)}}{\log(y_0)}\mathbf{1}_{\mathcal{G}_\ell^{\text{Rad}}}\int_T^{2T}\prod_{e^{1/|T|} < p \leq y_0}\left|1+\frac{f(p)}{p^{1/2+\sigma_k+it}}\right|^2dt\right)^q \\
    + \ \E&\left(\frac{1+(1-q)\sqrt{\log_2(y_0)}}{\log(y_0)} \ \mathbf{1}_{\mathcal{G}^{\text{Rad}}_\ell \text{ fails}}\int_T^{2T}\prod_{e^{1/|T|} < p \leq y_0}\left|1+\frac{f(p)}{p^{1/2+\sigma_k+it}}\right|^2dt\right)^q.
\end{align*}
The first expectation has already been handled in equation (\ref{good_event_gaussian}), so we only need to consider the second term. To do this, we follow the strategy of Harper in \cite{aharper201} to combine the good event and the failure event into one bound, which is achieved by a recursive procedure. For each $q$, we  apply H\"{o}lder's inequality to the second term in the above with exponents $\frac{1+q}{1-q}$ and $\frac{1+q}{2q}$ to get
\begin{align*}&\E\left(\frac{1+(1-q)\sqrt{\log_2(y_0)}}{\log(y_0)}\mathbf{1}_{\mathcal{G}^{\text{Rad}}_\ell \text{ fails}}\int_T^{2T}\prod_{e^{1/|T|} < p \leq y_0}\left|1+\frac{f(p)}{p^{1/2+\sigma_k+it}}\right|^2dt\right)^q \\ \leqslant &\left(\E\left[ \left(\frac{1+(1-q)\sqrt{\log_2(y_0)}}{\log(y_0)}\int_T^{2T}\prod_{e^{1/|T|} < p \leq y_0}\left|1+\frac{f(p)}{p^{1/2+\sigma_k+it}}\right|^2dt\right)^{\frac{1+q}{2}}\right]\right)^{\frac{2q}{1+q}}\left(\Prob\left(\mathcal{G}_\ell^{\text{Rad}} \text{ fails}\right)\right)^{\frac{1-q}{1+q}}.
\end{align*}
We notice that this has increased the size of the power of the integral when we take its expectation. This means that at each application of H\"{o}lder's inequality, we are getting $\frac{1-q}{2}$ closer to the $q=1$ moment of the integral, at which point one applies Fubini--Tonelli, while the fail events occur with probability bounded away from $1$ when raised to the $\frac{1-q}{1+q}$-th power. In particular, we see that for $0 < \delta < 1/6$
\begin{align*}
    \sup_{1-2\delta \leqslant q \leqslant 1-\delta}&\E\left(\frac{1+(1-q)\sqrt{\log_2(y_0)}}{\log(y_0)} \ \mathbf{1}_{\mathcal{G}^{\text{Rad}}_\ell \text{ fails}}\int_T^{2T}\prod_{e^{1/|T|} < p \leq y_0}\left|1+\frac{f(p)}{p^{1/2+\sigma_k+it}}\right|^2dt\right)^q \\
    \leqslant \sup_{1-\delta < \frac{1+q}{2} < 1-\delta/2}& \Biggr[\Bigg(\E\Bigg[\left(\frac{1+(1-q)\sqrt{\log_2(y_0)}}{\log(x)}\int_T^{2T}\prod_{e^{1/|T|} < p \leq y_0}\left|1+\frac{f(p)}{p^{1/2+\sigma_k+it}}\right|^2dt\right)^{\frac{1+q}{2}}\Bigg]\Bigg)^{\frac{2q}{1+q}} \\
    \cdot& \left(\Prob\left(\mathcal{G}_\ell^{\text{Rad}} \text{ fails}\right)\right)^{\frac{1-q}{1+q}}\Biggr].
\end{align*}
If one continues this process by replacing $q$ with $q'=\frac{1+q}{2q}$ many times, we will get that for $0 < \delta < \frac{1}{\sqrt{\log_2(y_0)}}$ and using the bound in equations (\ref{fail_event}) and \ref{good_event_gaussian},
\begin{align*}
    \E&\left(\frac{1+(1-q)\sqrt{\log_2(y_0)}}{\log(y_0)} \ \mathbf{1}_{\mathcal{G}^{\text{Rad}}_\ell \text{ fails}}\int_T^{2T}\prod_{e^{1/|T|} < p \leq y_0}\left|1+\frac{f(p)}{p^{1/2+\sigma_k+it}}\right|^2dt\right)^q \\ \ll& T^2\log(1/|T|) + \E\left[\left(\int_T^{2T}\prod_{e^{1/|T|} < p \leq y_0}\left|1+\frac{f(p)}{p^{1/2+\sigma_k+it}}\right|^2dt\right)^{1-\delta}\right].
\end{align*}
This is why it is very important that $C$ is taken as a large constant so we can use the recursive bound as many times as we need without exceeding the magnitude of the initial term. Then we exchange the order of integration and expectation, and find that by Holder's inequality
\begin{align*}&\E\left[\left(\frac{1}{\log(y_0)}\int_T^{2T}\prod_{e^{1/|T|} < p \leq y_0}\left|1+\frac{f(p)}{p^{1/2+\sigma_k+it}}\right|^2dt\right)^{1-\delta}\right] \\ \leqslant \ &\left(\frac{1}{\log(y_0)}\int_T^{2T}\E\left[\prod_{e^{1/|T|} < p \leq y_0}\left|1+\frac{f(p)}{p^{1/2+\sigma_k+it}}\right|^2\right]dt\right)^{1-\delta}  \ll  T^2,
\end{align*}
which is more than sufficient for what we want to prove. Thus, we have proved 
\begin{equation}
    \E\left[\left(\frac{1}{\log(y_0)}\int_T^{2T}\left|F_{y_0}(1/2+\sigma_{\ell}+it)\right|^2 dt\right)^q\right] \ll \left(\frac{T^{3-2q}\log(1/|T|)}{1+(1-q)\sqrt{\log_2(y_0)}}\right)^{q}. \label{complete_result}
\end{equation}
When we apply Markov's inequality with exponent $q=1/3$ and use equation (\ref{complete_result}), we obtain
\begin{align*}
    \Prob\left(\mathcal{P}_\ell^{(3)}(T)\right) \leqslant& \frac{\ell^{K/6}}{\ell^{4/3}}\E\left[\left(\frac{1}{\log(y_0)}\sum_{\substack{(\log_2(y_0))^{-K} \leqslant |T| \leqslant 1/1024\\ T \text{ dyadic}}}\frac{\text{sgn}(T)}{T^{2}} \int_{T}^{2T}\left|F_{y_0}(1/2+\sigma_\ell+it)\right|^2 dt\right)^{1/3}\right] \\
    \leqslant& \frac{\ell^{K/6}}{\ell^{4/3}}\sum_{\substack{(\log_2(y_0))^{-K} \leqslant |T| \leqslant 1/1024\\ T \text{ dyadic}}}\E\left[\left( \frac{1}{T^2 \log(y_0)}\int_{T}^{2T}\left|F_{y_0}(1/2+\sigma_\ell+it)\right|^2 dt\right)^{1/3}\right] \\
    \ll& \frac{\ell^{{K/6}}}{\ell^{4/3}}\sum_{\substack{(\log_2(y_0))^{-K} \leqslant |T| \leqslant 1/1024\\ T \text{ dyadic}}}\frac{1}{|T|^{2/3}}\left(\frac{|T|^{7/3}\log(1/|T|)}{\sqrt{\log_2(y_0)}}\right)^{1/3} \\
    \ll& \frac{\ell^{K/6}}{\ell^{4/3}(\log_2(y_0))^{1/6}} \ll \frac{1}{\ell^{4/3}},
\end{align*}
using that $\log_2(y_0) \asymp \ell^K$. This is clearly summable in $\ell$, so this is sufficient to conclude this section.
\section{Bounding the good events} \label{good_good}
We are now in position to show the convergence of $\Prob(\mathcal{B}_{1,\ell})$ since we have sufficiently strong almost sure bounds on the complement events. Collecting all the estimates together, we have for some $C>0$
\begin{align*}
    C\sum_{j^* < j \leqslant J}\frac{V_\ell(x_i,y_j;f)}{\ell^{K}} \leqslant& \frac{\ell\log(\ell)}{\ell^{K}}\sup_{j^* < j \leqslant J}\left( \mathcal{Q}_\ell^{(1)}(x_i,y_j;f) + \mathcal{Q}_\ell^{(2)}(x_i,y_j;f) +\mathcal{Q}_\ell^{(3)}(x_i,y_j;f)\right) \\ +& \frac{1}{\ell^K}\sum_{j^* < j \leqslant J} (\mathcal{D}_\ell(x_i,y_j;f) + \mathcal{C}_\ell^{(2)}(x_i,y_j;f)).
\end{align*}
Define the event
$$\mathcal{V}_\ell := \left\{\sup_{X_{\ell-1} < x_i \leqslant X_\ell}\sum_{1 \leqslant j \leqslant J}\frac{V_\ell(x_i,y_j;f)}{\ell^{K/2}}> \frac{T(\ell)}{C}\right\}.$$
\begin{lemma}
    The sum $\sum_{\ell \geqslant 1}\Prob(\mathcal{V}_\ell)$ converges. \label{Vconv}
\end{lemma}
\begin{proof}
Following Caich \cite{caich23}, using the triangle inequality, we see
$$\Prob(\mathcal{V}_\ell) \leqslant \Prob\left(\mathcal{Q}_\ell^{(1)}\right) + \Prob\left(\mathcal{Q}_\ell^{(2)}\right) + \Prob\left(\mathcal{Q}_\ell^{(3)}\right) + \Prob\left(\mathcal{D}_\ell\right)  + \Prob\left(\mathcal{C}_\ell^{(2)}\right),$$
and $T(\ell) \geqslant \ell^{24}$, as well as all our previous work in Lemmas \ref{D_conv}, \ref{C_2conv}, \ref{Q_conv} and \ref{Q*conv}, it is clear $\Prob(\mathcal{V}_\ell)$ is summable in $\ell$ (we get even more cancellation in the latter events).
\end{proof}
\begin{lemma}
    The sum $\sum_{\ell \geqslant 1}\Prob(\mathcal{B}_{1,\ell})$ converges. \label{B1conv}
\end{lemma}
\begin{proof}
We appeal to Proposition \ref{Caich_AH_in} to prove this lemma. To this end, we define the events
\begin{align*}
    \mathcal{E}_{\ell,i} :=& \left\{\sum_{j^* \leqslant j \leqslant J}V_\ell(x_i,y_j;f) \leqslant \frac{\ell^{K/2}T(\ell)}{C}\right\}, \\
    \mathcal{E}_\ell :=& \bigcap_{X_{\ell-1} < x_i \leqslant X_\ell}\mathcal{E}_i.
\end{align*}
Clearly we have $\Prob(\overline{\mathcal{E}}_\ell) = \Prob(\mathcal{V}_\ell)$. By the choice of $T(\ell)$, we proved in Lemma \ref{Vconv} that $\Prob(\mathcal{V}_\ell)$ is summable in $\ell$. This allows us to write
\begin{align*}
    \Prob(\mathcal{B}_{1,\ell}) =& \ \Prob\left(\sup_{X_{\ell-1} < x_i \leqslant X_\ell}\sum_{j^* \leqslant j \leqslant J}\frac{|S_{i,j}|}{(\log_2(x_i))^{3/4+\epsilon}} > 2\right) \\
    \leqslant& \ \Prob\left(\bigcup_{X_{\ell-1} < x_i \leqslant X_\ell}\left(\left\{\sum_{j^* \leqslant j \leqslant J}\frac{|S_{i,j}|}{(\log_2(x_i))^{3/4+\epsilon}} > 1\right\} \bigcap \{\mathcal{E}_{\ell}\}\right)\right) + \Prob(\mathcal{V}_\ell) \\
    \leqslant& \sum_{X_{\ell-1} < x_i \leqslant X_\ell}\Prob\left(\left\{\sum_{j^* \leqslant j \leqslant J}\frac{|S_{i,j}|}{(\log_2(x_i))^{3/4+\epsilon}} > 1\right\} \bigcap \{\mathcal{E}_{\ell,i}\}\right) + \Prob(\mathcal{V}_\ell).
\end{align*}
Applying Proposition \ref{Caich_AH_in}, we have
$$\Prob\left(\left\{\sum_{j^* \leqslant j \leqslant J}\frac{|S_{i,j}|}{(\log_2(x_i))^{3/4+\epsilon}} > 1\right\} \bigcap \{\mathcal{E}_{\ell,i}\}\right) \ll \exp\left(-\frac{C\ell^{3K/2+2K\epsilon}}{T(\ell)\ell^{K/2}}\right) \ll \exp\left(-C_1\ell^{K+2K\epsilon-24}\right)$$
for some $C_1 >0$. Then using $K\epsilon>24$, we sum over the $x_i$ in the interval $(X_{\ell-1},X_\ell]$. We then have
\begin{align*}
    \Prob(\mathcal{B}_{1,\ell}) \ll& \sum_{X_{\ell-1} < x_i \leqslant X_\ell} \exp\left(-C_1\ell^{K}\ell^{24}\right) + \Prob(\mathcal{V}_\ell) \\
    \ll& \exp\left(-\ell^K\left(C_1\ell^{24}-\frac{\log(2)}{\gamma}\right)\right) + \Prob(\mathcal{V}_\ell),
\end{align*}
which is summable in $\ell$.
\end{proof}
Given Lemmas \ref{PB_0k}, \ref{B1conv} and showing that $\Prob(\mathcal{H}_\ell')$ is summable in $\ell$ (done in Section \ref{euler_LOIL}) this completes the proof of Key Proposition \ref{key_prop}.
\subsection{Finishing the proof of Theorem \ref{ubtheorem}} \label{final_step}
\begin{proof}[Proof of Theorem \ref{ubtheorem}]
    From the triangle inequality, we have
\begin{equation*}
    |M_f(x)| \leqslant |M_f(x_{i-1})| + \max_{x_{i-1} < y \leqslant x_i}|M_f(y) - M_f(x_{i-1})|. 
\end{equation*}
For the first term, we can apply Key Proposition \ref{key_prop} and the second term is covered using Lemma \ref{test_lem}. Then we find that
\begin{align*}
    |M_f(x)| \ll& \frac{1}{\log(x_{i-1})} + (\log_2(x_{i-1}))^{3/4+\epsilon} \\
    \ll& (\log_2(x))^{3/4+\epsilon}.
\end{align*}
\end{proof}
\section{Lower Bound} \label{lower_bound_all}
In this section we prove that for any $\epsilon>0$,
\begin{equation*}
    \Prob\left(\max_{T_{k-1} < x \leqslant T_k} |M_f(x)|^2 \geqslant (\log_2(T_k))^{-1}\text{ i.o.} \right) = 1
\end{equation*}
for some appropriately chosen sequence $(T_k)_{k\geqslant1}$. This will be sufficiently to prove Theorem \ref{lbtheorem} for reasons that will be explored later in this proof. This will follow Hardy's work in a far more direct manner, with the differences arising from using Rademacher Euler products as opposed to Steinhaus ones. We do make a small change compared to that of Hardy, with us looking to bound the random Euler product slightly away from the critical line. This allows us to obtain a far better lower bound. We find that there is a deterministic contribution which appears, and can conclude from there. To upper bound the random contribution, we use Propositions \ref{HA1} and \ref{chernoff} (which are again simple adaptations of Hardy's work). We note that the imaginary shift taken here could likely be optimised further.\par
To proceed with the proof, let $\epsilon>0$ be fixed and choose the sequence $T_k = \exp(\exp(\lambda^k))$ for some $\lambda >1$ to be chosen later. Next, we notice that $\int_1^{T_k}\frac{dt}{t} = \log(T_k)$, and applying H\"{o}lder's inequality we see that
\begin{equation*}
    \max_{T_{k-1} < x \leqslant T_k}|M_f(x)|^2 \geqslant \frac{1}{\log(T_k)}\int_{T_{k-1}}^{T_k}\frac{|M_f(t)|^2}{t^{1+2\sigma_k}}dt
\end{equation*}
where $\sigma_k = \log_2(T_k)/\log(T_k)$. We want to complete the range of this integral at minimal cost, so we show that the tail
\begin{equation*}
    \frac{1}{\log(T_k)}\int_{T_k}^\infty\frac{\left|\sum_{\substack{n \leqslant t \\P(n) \leqslant T_k}}\frac{f(n)}{\sqrt{n}}\right|^2}{t^{1+2\sigma_k}}dt
\end{equation*}
and the initial part
$$\frac{1}{\log(T_k)}\int_1^{T_{k-1}}\frac{\left|\sum_{n \leqslant t}\frac{f(n)}{\sqrt{n}}\right|^2}{t^{1+2\sigma_k}} dt.$$
is almost surely small enough to be ignored. From Proposition \ref{prop_hyper}, we know that
$$\E\left(\left|\sum_{\substack{n \leqslant t \\P(n) \leqslant T_k}}\frac{f(n)}{\sqrt{n}}\right|^2\right) \ll \log(T_k).$$
Then we see after applying Markov's inequality, Fubini--Tonelli, Proposition \ref{prop_hyper} and Proposition \ref{euler_prod_result},
\begin{align*}
    &\Prob\left(\frac{1}{\log(T_k)}\int_{T_k}^\infty\frac{\left|\sum_{\substack{n \leqslant t \\ P(n) \leqslant T_k}}\frac{f(n)}{\sqrt{n}}\right|^2}{t^{1+2\sigma_k}}dt > (\log_2(T_k))^{-2}\right) \\ \leqslant& \frac{(\log_2(T_k))^2}{\log(T_k)}\int_{T_k}^{\infty}\frac{\E\left(\left|\sum_{\substack{n \leqslant t\\P(n) \leqslant T_k}}\frac{f(n)}{\sqrt{n}}\right|^2\right)}{t^{1+2\sigma_k}}dt 
    \ll \frac{\log_2(T_k)}{\log(T_k)} \ll \lambda^k e^{-\lambda^k}
\end{align*}
which is certainly summable in $k$, so we conclude using the first Borel--Cantelli lemma. This allows us to discard the tail of this integral. For the initial contribution of the integral, we proceed in a similar fashion by using Proposition \ref{prop_hyper} and Fubini--Tonelli,
$$\E\left(\int_1^{T_{k-1}}\frac{\left|\sum_{n \leqslant t}\frac{f(n)}{\sqrt{n}}\right|^2}{t^{1+2\sigma_k}}\right) \leqslant \int_1^{T_{k-1}}\frac{\log(t)}{t^{1+2\sigma_k}} dt \ll \int_0^{\log(T_{k-1})} u \ du \ll (\log(T_{k-1}))^2.$$
Then by Markov's inequality, we obtain that
\begin{align*}
    &\Prob\left(\frac{1}{\log(T_k)}\int_{1}^{T_{k-1}}\frac{\left|\sum_{\substack{n \leqslant t \\ P(n) \leqslant T_k}}\frac{f(n)}{\sqrt{n}}\right|^2}{t^{1+2\sigma_k}}dt > (\log_2(T_k))^{-2}\right) \\ \ll& \frac{(\log(T_{k-1}))^2(\log_2(T_k))^2}{\log(T_k)} \ll \lambda^{2k}\exp\left(2\lambda^{k-1}-\lambda^k\right).
\end{align*}
which is summable in $k$ given $\lambda > 2$. Since we will be choosing $\lambda$ as arbitrarily large, this condition is achieved. We then conclude by the first Borel--Cantelli lemma. \par
Overall, we have obtained that for sufficiently large $k$,
\begin{equation}
    \max_{x \in [T_{k-1},T_k]}|M_f(x)|^2 \geqslant \frac{1}{\log(T_k)}\int_1^\infty \frac{\left|\sum_{\substack{n \leqslant t\\P(n) \leqslant T_k}}\frac{f(n)}{\sqrt{n}}\right|^2}{t^{1+2\sigma_k}}dt - \frac{C}{(\log_2(T_k))^2}. \label{small_tail}
\end{equation}
for some $C>0$ constant. Next, we appeal to Proposition \ref{HA1} to obtain the following:
\begin{align*}
    \frac{1}{\log(T_k)}\int_1^\infty \frac{\left|\sum_{\substack{n \leqslant t\\P(n) \leqslant T_k}}\frac{f(n)}{\sqrt{n}}\right|^2}{t^{1+2\sigma_k}}dt =& \frac{1}{2\pi \log(T_k)}\int_{-\infty}^{\infty}\frac{|F_{T_k}(1/2+\sigma_k+it)|^2}{|\sigma_k+it|^2}dt \\
    \geqslant& \frac{(1+o(1))\log(T_k)}{2 \pi (\log_2(T_k))^2}\int_{\frac{1}{2\log(T_k)}}^{\frac{3}{2\log(T_k)}}|F_{T_k}(1/2+\sigma_k+it)|^2dt.
\end{align*}
This differs from the proof Hardy uses since the Rademacher Euler product is far more fiddly than the Steinhaus one, so in order to obtain sharp bounds, one should investigate over the whole $[-1/2,1/2]$ interval in a similar manner to that of Harper for his lower bounds in \cite{harper202}. However, since we do not expect to obtain sharp bounds here, we refrain from doing this. We investigate the behaviour of the random Euler product in the interval $[1/2\log(T_k),3/2\log(T_k)]$ interval since this captures the behaviour we want to exploit to prove Theorem \ref{lbtheorem} and for tidiness.\par
At this point, we observe that $\int_{\frac{1}{2\log(T_k)}}^{\frac{3}{2\log(T_k)}}\log(T_k)dt$ is a probability measure, so by Jensen's inequality,
\begin{align*}
    \frac{(1+o(1))\log(T_k)}{2\pi(\log_2(T_k))^2}&\int_{\frac{1}{2\log(T_k)}}^{\frac{3}{2\log(T_k)}}|F_{T_k}(1/2+\sigma_k+it)|^2dt \\ \geqslant \frac{1+o(1)}{2\pi (\log_2(T_k))^2}&\exp\left(\int_{\frac{1}{2\log(T_k)}}^{\frac{3}{2\log(T_k)}} 2\sum_{p \leqslant T_k}\Re\log\left(1+\frac{f(p)}{p^{1/2+\sigma_k+it}}\right)\log(T_k) dt\right) \\
    \geqslant \frac{1+o(1)}{2\pi (\log_2(T_k))^2}&\exp\left(2\log(T_k)\sum_{p \leqslant T_k}\int_{\frac{1}{2\log(T_k)}}^{\frac{3}{2\log(T_k)}}\Re\left(\frac{f(p)}{p^{1/2+\sigma_k+it}} - \frac{1}{2p^{1+2\sigma_k+2it}} + O(p^{-3/2})\right) dt \right).
\end{align*}
Then using that $p^{-3/2}$ is summable over the primes, we can lower bound the final expression by
\begin{equation}
    \frac{c}{(\log_2(T_k))^2}\exp\left(2\log(T_k)\sum_{p \leqslant T_k}\int_{\frac{1}{2\log(T_k)}}^{\frac{3}{2\log(T_k)}}\Re\left(\frac{f(p)}{p^{1/2+\sigma_k+it}} - \frac{1}{2p^{1+2\sigma_k+2it}}\right)dt \right) \label{pre_LOIL}
\end{equation}
for some $c>0$. For the two integrals, we have the following quantities in terms of $\ell$:
\begin{align*}
    \int_{\frac{1}{2\log(T_k)}}^{\frac{3}{2\log(T_k)}} \Re p^{-it} dt =& \frac{2}{\log(p)}\cos\left(\frac{\log(p)}{\log(T_k)}\right)\sin\left(\frac{\log(p)}{2\log(T_k)}\right), \\ \int_{\frac{1}{2\log(T_k)}}^{\frac{3}{2\log(T_k)}}\Re p^{-2it}dt =& \frac{1}{\log(p)}\cos\left(\frac{2\log(p)}{\log(T_k)}\right)\sin\left(\frac{\log(p)}{\log(T_k)}\right) = \frac{\cos\left(\frac{2\log(p)}{\log(T_k)}\right)}{\log(T_k)} + O\left(\frac{(\log(p))^2}{(\log(T_k))^3}\right).
\end{align*}
where the final equality was achieved by Taylor expanding the sine function. Consequently, we find a lower bound on equation (\ref{pre_LOIL}),
\begin{align*}
    \geqslant \frac{c}{(\log_2(T_k))^2} \exp\Bigg(&2\log(T_k)\Bigg(\sum_{p \leqslant T_k}\frac{f(p)}{p^{1/2+\sigma_k}}\frac{2\cos(\frac{\log(p)}{\log(T_k)})\sin(\frac{\log(p)}{2\log(T_k)})}{\log(p)} \\
    \quad &- \frac{\cos(\frac{2\log(p)}{\log(T_k)})}{2p^{1+2\sigma_k}\log(T_k)} + O\left(\frac{(\log(p))^2}{p(\log(T_k))^3}\right) \Bigg)\Bigg) \\
    \geqslant \frac{c'}{(\log_2(T_k))^2} \exp&\left(2\left(\sum_{p \leqslant T_k}\frac{f(p)\cos(\frac{\log(p)}{\log(T_k)})}{p^{1/2+\sigma_k}}\frac{2\log(T_k)}{\log(p)}\sin\left(\frac{\log(p)}{2\log(T_k)}\right) 
    - \frac{\cos(\frac{2\log(p)}{\log(T_k)})}{2p^{1+2\sigma_k}} \right)\right),
\end{align*}
for some $c' > 0$. To obtain the final inequality, we used $\sum_{p \leqslant T_k}\frac{(\log(p))^2}{p} \ll (\log(T_k))^2$. We want to remove the deterministic contribution from the above. This is achieved by Taylor expanding the exponential and observing that $p^{-3/2}$ is summable in $p$. We have
\begin{align*}
    \exp&\left(2\left(\sum_{p \leqslant T_k}\frac{f(p)\cos(\frac{\log(p)}{\log(T_k)})}{p^{1/2+\sigma_k}}\frac{2\log(T_k)}{\log(p)}\sin\left(\frac{\log(p)}{2\log(T_k)}\right) 
    - \frac{\cos(\frac{2\log(p)}{\log(T_k)})}{2p^{1+2\sigma_k}} \right)\right) \\
    =& 1 + \sum_{p \leqslant T_k}\frac{f(p)\cos(\frac{\log(p)}{\log(T_k)})}{p^{1/2+\sigma_k}}\frac{4\log(T_k)}
    {\log(p)}\sin\left(\frac{\log(p)}{2\log(T_k)}\right)  \numberthis \label{taylor_expand}\\ 
    +& \sum_{p \leqslant T_k}\left[\left(\frac{2\cos^2(\frac{\log(p)}{\log(T_k)})}{p^{1+2\sigma_k}}\left(\frac{2\log(T_k)}{\log(p)}\sin\left(\frac{\log(p)}{2\log(T_k)}\right)\right)^2 -\frac{\cos(\frac{2\log(p)}{\log(T_k)})}{p^{1+2\sigma_k}}\right) +O(p^{-3/2})\right].  
\end{align*}
The final term is clearly summable. We use the trigonometric identity $2\cos^2(x) - \cos(2x) = 1$ and the Taylor expansion $\frac{1}{u^2}\sin^2(u) = 1 - \frac{u^2}{3} + O(u^4)$, which together imply 
$$2\cos^2(x)\left(\frac{\sin(u)}{u}\right)^2 - \cos(2x) = 1 - \frac{2\cos^2(x)u^2}{3} + O(u^4).$$
Applying this to the Taylor expansion above yields
\begin{align*}
    &\sum_{p \leqslant T_k}\left(\frac{2\cos^2(\frac{\log(p)}{\log(T_k)})}{p^{1+2\sigma_k}}\left(\frac{2\log(T_k)}{\log(p)}\sin\left(\frac{\log(p)}{2\log(T_k)}\right)\right)^2 -\frac{\cos(\frac{2\log(p)}{\log(T_k)})}{p^{1+2\sigma_k}}\right) \\
    =& \sum_{p \leqslant T_k}\left[\frac{1}{p^{1+2\sigma_k}}-\left(\frac{2\cos^2(\frac{\log(p)}{\log(T_k)})}{3p^{1+2\sigma_k}}\left(\frac{\log(p)}{2\log(T_k)}\right)^2+O\left(\frac{(\log(p))^4}{p^{1+2\sigma_k}(\log(T_k))^4}\right)\right)\right],
\end{align*}
To deal with the second term, we take the bound $\cos^2(x) \leqslant 1$, and obtain that this can be bounded above by
$$\frac{1}{6}\sum_{p \leqslant T_k}\frac{(\log(p))^2}{p^{1+2\sigma_k}(\log(T_k))^2} + O\left(\frac{(\log(p))^4}{p^{1+2\sigma_k}(\log(T_k))^4}\right)$$
which is at most constant when one observes that $\sum_{p \leqslant x}\frac{(\log(p))^n}{p} = O\left((\log(x))^n\right)$. We have
$$\sum_{p \leqslant T_k}\frac{1}{p^{1+2\sigma_k}} \geqslant \sum_{p \leqslant \exp\left(\frac{\log(T_k)}{(\log_2(T_k))^2}\right)}\frac{1}{p^{1+2\sigma_k}} \gg \log_2(T)$$
We will show that this term dominates the random contribution. For this purpose, we follow the approach used in Section \ref{euler_LOIL}, which is very similar to Section 2.7 of Hardy's work \cite{shardy23}, which we detail now. Note that we will not obtain such sharp bounds as Hardy since we have even sparser test points. We will prove that
\begin{equation}
    \Prob\left(\sum_{p \leqslant T_k}\frac{2f(p)\cos(\frac{\log(p)}{\log(T_k)})}{p^{1/2+\sigma_k}}\frac{2\log(T_k)}{\log(p)}\sin\left(\frac{\log(p)}{2\log(T_k)}\right) > (\log_2(T_k))^{0.99}\right) \label{exceptional_random}
\end{equation}
is summable in $k$ by the first Borel--Cantelli lemma. First we observe that one has $\left|\frac{\sin(u)}{u}\right| \leqslant 1$ for any $u \in \mathbb{R}$ and $|\cos(x)| \leqslant 1$ for any $x \in \mathbb{R}$. This means that we can upper bound the probability in equation (\ref{exceptional_random}) by
$$\Prob\left(\sum_{p \leqslant T_k} \frac{2f(p)}{p^{1/2+\sigma_k}} > (\log_2(T_k))^{0.99}\right).$$
Then we can apply the Proposition \ref{chernoff} with $x = \frac{(\log_2(T_k))^{0.49}}{2}$, and we see that
$$\Prob\left(\sum_{p \leqslant T_k} \frac{2f(p)}{p^{1/2+\sigma_k}} > (\log_2(T_k))^{0.49}\right) \leqslant \exp\left(-\frac{(\log_2(T_k))^{0.98}}{8}+B(\log_2(T_k))^{0.97}\right),$$
where $B>0$ is some absolute constant. This is clearly summable in $k$ for sufficiently large $\lambda$. Returning to equation (\ref{taylor_expand}), then we have the following almost surely lower bound:
$$\geqslant c_1\log_2(T_k),$$
for some small $c_1 > 0$. This lower bound is achieved for any sufficiently large $k$, so we have that for some $0 < c_2 < c_1c'$
$$\max_{T_{k-1} < t \leqslant T_k}|M_f(t)|^2 \geqslant \frac{c'c_1\log_2(T_k)}{(\log_2(T_k))^2} - \frac{C}{(\log_2(T_k))^2} > \frac{c_2}{\log_2(T_k)}.$$
\section{Proof of Theorem \ref{sharp_restricted_ub}} \label{mastro_variant}
This proof follows from the work of Mastrostefano \cite{mastrostefano2022almost}. The key point is that when we have one prime factor, then we do not need to apply the union bound in equation (\ref{Q*ineq}) which saves a factor of $\sqrt{\log_2(x)}$ overall. We will outline the general steps (which are slightly different because one needs to recover the slow variation bound we prove in Lemma \ref{test_lem}). We can keep the definitions of the $x_i$ and $X_\ell$ the same as in the proof of Theorem \ref{ubtheorem} for convenience, as well as $T(\ell)= \ell^{24}$. We will sketch the proof given by Mastrostefano (with minor adjustments), and input the results we found in Section \ref{tilted_prob}. In this section, we set
$$S_f(x) := \sum_{\substack{n \leqslant x \\P(n) > \sqrt{x}}}\frac{f(n)}{\sqrt{n}}.$$
With $x_i = [e^{i^{\gamma}}]$ chosen in Lemma \ref{test_lem} and $X_\ell = \exp\left(2^{\ell^K}\right)$ with $K = \left\lfloor\frac{25}{\epsilon}\right\rfloor$. Note that we are keeping the definitions of $\gamma$ and $\epsilon$ separate, whereas Mastrostefano takes $\gamma = \epsilon$. It is likely that this distinction is not required, but for simplicity and consistency with the paper so far, we keep them separate. We define the event
\begin{equation}
    \widetilde{\mathcal{A}}_\ell := \left\{\sup_{X_{\ell -1} < x_{i-1} \leqslant X_\ell}\sup_{x_{i-1} < x \leqslant x_i}\frac{\left|S_f(x)\right|}{(\log_2(x))^{1/4+\epsilon}} > 6\right\}. \label{mastro_main}
\end{equation}
Then one observes that $\widetilde{\mathcal{A}}_\ell \subset \widetilde{\mathcal{B}}_\ell \cup \widetilde{\mathcal{C}}_\ell \cup \widetilde{\mathcal{D}}_\ell$ with
\begin{align}
    \widetilde{\mathcal{B}}_\ell :=& \left\{\sup_{X_{\ell-1} < x_{i-1} \leqslant X_\ell}\frac{|S_f(x_{i-1})|}{(\log_2(x_{i-1}))^{1/4+\epsilon}} > 2\right\}; \\
    \widetilde{\mathcal{C}}_\ell :=& \left\{\sup_{X_{\ell-1} < x_{i-1} \leqslant X_\ell} \frac{1}{(\log_2(x_{i-1}))^{1/4+\epsilon}}\sup_{x_{i-1} < x \leqslant x_i}\left|\sum_{\substack{n \leqslant x_{i-1} \\ \sqrt{x_{i-1}} < P(n) \leqslant \sqrt{x}}}\frac{f(n)}{\sqrt{n}}\right| > 2\right\}; \\
    \widetilde{\mathcal{D}}_\ell :=& \left\{\sup_{X_{\ell-1} < x_{i-1} \leqslant X_\ell}\frac{1}{(\log_2(x_{i-1}))^{1/4+\epsilon}}\sup_{x_{i-1} < x \leqslant x_i}\left|\sum_{\substack{x_{i-1} < n \leqslant x \\ P(n) > \sqrt{x}}}\frac{f(n)}{\sqrt{n}} \right| > 2\right\}.
\end{align}
By the triangle inequality, we see that $\Prob(\widetilde{\mathcal{A}}_\ell) \leqslant \Prob(\widetilde{\mathcal{B}}_\ell) + \Prob(\widetilde{\mathcal{C}}_\ell) + \Prob(\widetilde{\mathcal{D}}_\ell)$. The events $\widetilde{\mathcal{C}}_\ell$ and $\widetilde{\mathcal{D}}_\ell$ are studied to develop an analog for Lemma \ref{test_lem}, whereas $\widetilde{\mathcal{B}}_\ell$ is the where the majority of the work is needed. \par
The bound for $\widetilde{\mathcal{C}}_\ell$ follows in a very similar fashion to the proof of Lemma \ref{Q_conv} (without the presence of the integral smoothing). By a suitable adaptation of Lemma \ref{Y_sub}, the sum used in the definition of $\widetilde{\mathcal{C}}_\ell$ is a submartingale with respect to the filtration $\mathcal{F}_{n,1} := \sigma(\{f(p):p \leqslant \sqrt{n}\})$, so we can apply Proposition \ref{Lr-Doob} to it. In all, we see that after the union bound, Chebychev's inequality and Doob's $L^2$ inequality
\begin{align*}
    \Prob(\widetilde{\mathcal{C}}_\ell) \ll& \sum_{X_{\ell-1} < x_{i-1} \leqslant X_\ell}\frac{1}{(\log_2(x_{i-1}))^{1+4\epsilon}}\E\left(\sup_{x_{i-1} < x \leqslant x_i}\left|\sum_{\substack{n \leqslant x_{i-1} \\ \sqrt{x_{i-1}} < P(n) \leqslant \sqrt{x}}}\frac{f(n)}{\sqrt{n}}\right|^4\right) \\
    \ll& \sum_{X_{\ell-1} < x_{i-1} \leqslant X_\ell}\frac{1}{(\log_2(x_{i-1}))^{1+4\epsilon}}\sup_{x_{i-1} < x \leqslant x_i}\E\left|\sum_{\substack{n \leqslant x_{i-1} \\ \sqrt{x_{i-1}} < P(n) \leqslant \sqrt{x}}}\frac{f(n)}{\sqrt{n}}\right|^4 \\
    \ll& \sum_{X_{\ell-1} < x_{i-1} \leqslant X_\ell}\frac{(\log(x_i))^6}{i^2}
\end{align*}
which is certainly summable in $\ell$ with our choice of $\gamma \leqslant 1/1000 \ (< 1/6)$. \par
Next we proceed to handling $\Prob(\widetilde{\mathcal{D}}_\ell)$, which follows by partial summation noticing that Mastrostefano's method (in Section 4.2 of his paper) allows for an even sharper bound than simply $\sqrt{x_{i-1}}(\log\log(x_{i-1}))^{1/4+\epsilon}$ (for instance $\frac{\sqrt{x_{i-1}}(\log_2(x_{i-1}))^{1/4+\epsilon}}{\log(x_{i-1})}$) by noticing all one has to do is choose $\gamma$ in his proof slightly smaller (in a way which is certainly satisfied by our choice of $\gamma \leqslant 1/1000$). Then one can apply partial summation, and we obtain
$$\sup_{x_{i-1} < x \leqslant x_i}\left|\sum_{\substack{x_{i-1} < n \leqslant x \\ P(n) > \sqrt{x}}}\frac{f(n)}{\sqrt{n}} \right| \ll (\log_2(x_{i-1}))^{1/4+\epsilon}$$
almost surely (one can find a sharper bound here, but we will not need it for our purposes). \par
This leaves us to handle the final $\widetilde{\mathcal{B}}_\ell$ term. Unsurprisingly, this is again requires the most work, and we follow very similar steps as in the proof of Theorem \ref{ubtheorem}, using
$$\widetilde{V}(x_i) = \sum_{\sqrt{x_i} < p \leqslant x_i}\frac{1}{p}\left|\sum_{\substack{m \leqslant x_i/p \\ P(n) < p}}\frac{f(n)}{\sqrt{n}}\right|^2.$$
Comparing to the proof of Theorem \ref{ubtheorem} and the prime splitting arguments there, this is similar to handling $V_\ell(x_i,y_j,f)$ for a single $j$ value. This allows us to follow the manipulations which appeared in Section \ref{sum_decomp}, and we see that
\begin{align*}
    \widetilde{V}(x_i) \ll& \sum_{\sqrt{x_i} < p \leqslant x_i}\frac{\mathcal{X}}{p^2}\int_p^{p(1+1/\mathcal{X})}\left|\sum_{\substack{n \leqslant x_i/t \\ P(n) < p}}\frac{f(n)}{\sqrt{n}}\right|^2dt \\ \numberthis \label{alt_var}
    +& \sum_{\sqrt{x_i} < p \leqslant x_i}\frac{\mathcal{X}}{p^2}\int_p^{p(1+1/\mathcal{X})}\left|\sum_{\substack{x_i/t  < n \leqslant x_i/p \\ P(n) < p}}\frac{f(n)}{\sqrt{n}}\right|^2dt.
\end{align*}
We can bound the second sum in a near identical fashion to the way we handled Lemma \ref{D_conv}. Note this will be a much shorter outer sum than we encountered there in Lemma \ref{D_conv}, so we can take $r$ and $\mathcal{X}$ as we did there and it will still be summable in $\ell$. To handle the first sum, we look to Section \ref{Q_good}, and this will follow after some minor changes; namely we define a normalised \textit{submartingale} sequence with respect to filtration $\mathcal{F}_n$ opposed to the normalised \textit{supermartingale} sequence used to prove Theorem \ref{ubtheorem}.  \par
To deal with the first term, we switch the order of summation and integration like we did in Section \ref{sum_decomp}. Then the first term is
$$\int_{\sqrt{x_i}}^{x_i(1+1/\mathcal{X})}\sum_{\frac{t}{1+1/\mathcal{X}} < p \leqslant t}\frac{\mathcal{X}}{p^2}\left|\sum_{n \leqslant \frac{x_i}{t}}\frac{f(n)}{\sqrt{n}}\right|^2dt.$$
Using equation (\ref{p_bound}) and the substitution $z:=x_i/t$, we have (using $z < x_i$ on the range of the transformed integral $[0,\sqrt{x_i}]$)
$$\ll \frac{1}{\log(x_i)}\int_0^{\sqrt{x_i}}\left|\sum_{\substack{n \leqslant z \\ P(n) \leqslant x_i}}\frac{f(n)}{\sqrt{n}}\right|^2 \frac{dz}{z^{1+\frac{1}{2\log(X_\ell)}}}.$$
We then complete the range of the integral, and then we are then in a position to apply Proposition \ref{HA1}. We have the upper bound
\begin{align*}
    \ll& \frac{1}{\log(x_i)}\int_{-\infty}^{\infty}\left|\frac{F_{x_i}(1/2+\frac{1}{\log(X_\ell)}+it)}{\frac{1}{\log(X_\ell)}+it}\right|^2 dt \\ \leqslant& \frac{1}{\log(x_i)}\left(\frac{\log(x_i)}{\log(X_{\ell-1})}\right)^{(\ell-1)^{-K}}\int_{-\infty}^{\infty}\left|\frac{F_{x_i}(1/2+\frac{1}{\log(X_\ell)}+it)}{\frac{1}{\log(X_\ell)}+it}\right|^2dt. \numberthis \label{submart}
\end{align*}
The introduction of the supremum allows for the use of Proposition \ref{Doob_max} later. For now, we show that $Y_{x_i}$ is a submartingale with respect to the filtration $\mathcal{F}_{i,2} := \sigma(\{f(p): p \leqslant x_i\})$ where we define 
$$Y_{x_i} :=\frac{1}{\log(x_i)}\left(\frac{\log(x_i)}{\log(X_{\ell-1})}\right)^{(\ell-1)^{-K}}\int_{-\infty}^{\infty}\left|\frac{F_{x_i}(1/2+\frac{1}{\log(X_\ell)}+it)}{\frac{1}{\log(X_\ell)}+it}\right|^2dt.$$
The first two properties are simple to observe, so we look to the third condition. Then we see that
\begin{align*}
    \E(Y_{x_i}|\mathcal{F}_{i-1,2}) =&  \frac{1}{\log(x_{i-1})}\frac{\log(x_{i-1})}{\log(x_{i})}\left(\frac{\log(x_{i-1})}{\log(X_{\ell-1})}\right)^{(\ell-1)^{-K}}\left(\frac{\log(x_{i})}{\log(x_{i-1})}\right)^{(\ell-1)^{-K}} \\ \times& \int_{-\infty}^{\infty}\left|\frac{F_{x_{i-1}}(1/2+\frac{1}{\log(X_\ell)}+it)}{\frac{1}{\log(X_\ell)}+it}\right|^2\E\prod_{x_{i-1} < p \leqslant x_i}\left|1+\frac{f(p)}{p^{1/2+1/\log(X_\ell)+it}}\right|^2dt.
\end{align*}
By Proposition \ref{euler_prod_result}, we have that the expectation has size
$$\E\prod_{x_{i-1} < p \leqslant x_i}\left|1+\frac{f(p)}{p^{1/2+1/\log(X_\ell)+it}}\right|^2 = \frac{\log(x_i)}{\log(x_{i-1})} = \exp\left(\frac{\gamma}{i} + O\left(\frac{1}{i^2}\right)\right),$$
where we have used the $\sum_{p \leqslant x_i}\frac{1}{p^{1+2/\log(X_\ell)}} = \log_2(x_i) + O(1)$ (since $x_i < X_\ell$). From this we see
$$\left(\frac{\log(x_{i})}{\log(x_{i-1})}\right)^{(\ell-1)^{-K}} =\exp\left(\frac{\log(2)}{i\log(i)} + O\left(\frac{1}{i^2\log(i)}\right)\right) \geqslant 1.$$
In particular, we see that
$$\E(Y_{x_i}|\mathcal{F}_{i-1,2}) \geqslant Y_{x_{i-1}},$$
which shows that $Y_{x_i}$ is a submartingale. \par
Next we introduce our conditioning, where we condition on the event
$$\Sigma_\ell := \left\{\frac{1}{\log(X_{\ell-1})}\int_{-\infty}^{\infty}\left|\frac{F_{X_{\ell-1}}(1/2+1/\log(X_\ell)+it)}{\frac{1}{\log(X_\ell)}+it}\right|^2dt \leqslant \frac{(T(\ell))^{1/2}}{\ell^{K/2}}\right\}.$$
For the complement event, this was already covered in Section \ref{ce}. From this, we obtain after the various manipulations required
$$\Prob(\overline{\Sigma_\ell}) \ll \ell^{4/3},$$
which is certainly summable in $\ell$. This leaves us with controlling the event on the conditioning (and then appealing to the Azuma--Hoeffding inequality). Then by Proposition \ref{Doob_max}, we have that
$$\Prob\left(\sup_{X_{\ell-1} < x_i \leqslant X_{\ell}}Y_{x_i} > \frac{CT(\ell)}{\ell^{K/2}} \bigg| \Sigma_\ell\right) \ll \frac{\ell^{K/2}}{T(\ell)}\E(Y_{x_I}|\Sigma_\ell),$$
where $x_I$ is the largest $x_i \leqslant X_\ell$. We observe that
\begin{align*}
     &\Prob\left(\sup_{X_{\ell-1} < x_i \leqslant X_{\ell}}Y_{x_i} > \frac{cT(\ell)}{\ell^{K/2}}\right)\\\ll \frac{\ell^{K/2}}{T(\ell)}\E&\left(\frac{1}{\log(X_{\ell-1})}\int_{-\infty}^{\infty}\left|\frac{F_{X_{\ell-1}}(1/2+1/\log(X_\ell)+it)}{\frac{1}{\log(X_\ell)}+it}\right|^2dt \bigg| \Sigma_\ell\right).
\end{align*}
By our conditioning, this expression has size $(T(\ell))^{-1/2}$ which is again summable in $\ell$. All that is now left is to apply the Proposition \ref{standard-AH} to bound the good event. We define the event
$$\widetilde{\mathcal{E}}_\ell' := \left\{\sup_{X_{\ell-1} < x_{i-1} \leqslant X_\ell}\widetilde{V}(x_i) \leqslant \frac{T(\ell)}{C\ell^{K/2}}\right\}.$$
Then we can proceed in a very similar fashion to that of Lemma \ref{B1conv} by decomposing $\widetilde{\mathcal{B}}_\ell$ using conditioning on $\widetilde{\mathcal{E}}_\ell'$:
\begin{align*}
    \Prob(\widetilde{\mathcal{B}}_\ell) \leqslant& \Prob\left(\left\{\sup_{X_{\ell-1} < x_{i-1} \leqslant X_\ell}\frac{|S_f(x_{i-1})|}{(\log_2(x_{i-1}))^{1/4+\epsilon}} > 2\right\} \bigcap \left\{\widetilde{\mathcal{E}}_\ell'\right\}\right) \\  +& \Prob\left(\left\{\sup_{X_{\ell-1} < x_{i-1} \leqslant X_\ell}\frac{|S_f(x_{i-1})|}{(\log_2(x_{i-1}))^{1/4+\epsilon}} > 2\right\} \bigcap \left\{\overline{\widetilde{\mathcal{E}}_\ell'}\right\}\right).
\end{align*}
The second of these have already been dealt with, so we look to the first term. By the union bound and Proposition \ref{standard-AH}, we have
\begin{align*}
    &\Prob\left(\left\{\sup_{X_{\ell-1} < x_{i-1} \leqslant X_\ell}\frac{|S_f(x_{i-1})|}{(\log_2(x_{i-1}))^{1/4+\epsilon}} > 2\right\} \bigcap \left\{\widetilde{\mathcal{E}}_\ell'\right\}\right)\\ \leqslant& \sum_{X_{\ell-1}< x_{i-1} \leqslant X_\ell}\exp\left(-\frac{C(\log_2(x_i))^{1/2+2K\epsilon}\ell^{K/2}}{T(\ell)}\right).
\end{align*}
Summing over the $\ell$ then shows that this is further bounded above by
$$\ll \exp\left(\ell^K\left(-C\ell^{24}+\frac{\log(2)}{\gamma}\right)\right)$$
which is clearly summable in $\ell$.
\printbibliography[title= References]

\end{document}